\documentclass[dvips,preprint]{imsart}
\RequirePackage[OT1]{fontenc} \RequirePackage{amsthm,amsmath}
\RequirePackage[numbers]{natbib}
\RequirePackage[colorlinks,citecolor=blue,urlcolor=blue]{hyperref}
\usepackage{amsmath}
\usepackage{amssymb}
\usepackage{amsmath}
\usepackage{amsfonts}
\usepackage[american]{babel}
\usepackage{graphicx}
\usepackage{flafter}
\usepackage[section]{placeins}
\usepackage{float}
\usepackage{makecell}
\startlocaldefs

\def\E{{\mathbb E }}

\def\goin{\to\infty}
\def\Bbb E{\mathbb{E}}
\def\Bbb R{\mathbb{R}}
\newtheorem{definition}{Definition}

\newtheorem{lemma}{Lemma}
\newtheorem{theorem}{Theorem}
\newtheorem{corollary}{Corollary}

\makeatletter \@addtoreset{equation}{section}

\makeatother

\font\tencmmib=cmmib10 \skewchar\tencmmib '60
\newfam\cmmibfam
\textfont\cmmibfam=\tencmmib

\font\tenmsb=msbm10 
\def\Bbb#1{\hbox{\tenmsb#1}}

\def\bbox{\quad\hbox{\vrule \vbox{\hrule \vskip2pt \hbox{\hskip2pt
\vbox{\hsize=1pt}\hskip2pt} \vskip2pt\hrule}\vrule}}
\def\lessim{\ \lower4pt\hbox{$
\buildrel{\displaystyle <}\over\sim$}\ }
\def\gessim{\ \lower4pt\hbox{$\buildrel{\displaystyle >}
\over\sim$}\ }

\def\goin{\to \infty}
\def\go0{\to 0}

\def\leftitem#1{\item{\hbox to\parindent{\enspace#1\hfill}}}

\def\qed{{$\hfill \bbox$}}

\def\sg{\sigma}

\def\sg2{\sigma^2}

\def\__{_{\infty}}
\numberwithin{equation}{section} \theoremstyle{plain}

\newcommand{\1}{{\rm 1}\kern-0.24em{\rm I}}
\def\E{\mathbb E}
\def\R{\mathbb R}
\newtheorem{assumption}{Assumption}

\newtheorem{rema}{Remark}

\endlocaldefs

\begin{document}

\begin{frontmatter}
\title{Normal approximation and concentration 
%for squared Hilbert-Schmidt norm errors 
of spectral projectors of sample covariance}\runtitle{Asymptotics and concentration of spectral projectors}

\begin{aug}
\author{\fnms{Vladimir} \snm{Koltchinskii}\thanksref{t1}\ead[label=e1]{vlad@math.gatech.edu}} and 
\author{\fnms{Karim} \snm{Lounici}\thanksref{m1}\ead[label=e2]{klounici@math.gatech.edu}}
%\thankstext{t1}{Supported in part by NSF grants
%DMS-09-06880 and CCF-0808863}
\thankstext{t1}{Supported in part by NSF Grants DMS-1207808, CCF-0808863 and CCF-1415498}
\thankstext{m1}{Supported in part by Simons Grant 315477 and NSF Career Grant, DMS-1454515}
\runauthor{V. Koltchinskii and K. Lounici}

\affiliation{Georgia Institute of Technology\thanksmark{m1}}

\address{School of Mathematics\\
Georgia Institute of Technology\\
Atlanta, GA 30332-0160\\
\printead{e1}\\
% \phantom{E-mail:\ }
\printead*{e2}
}
\end{aug}

\begin{abstract}
Let $X,X_1,\dots, X_n$ be i.i.d. Gaussian random variables in a separable Hilbert space ${\mathbb H}$ with zero mean and covariance operator $\Sigma={\mathbb E}(X\otimes X),$ and let 
$\hat \Sigma:=n^{-1}\sum_{j=1}^n (X_j\otimes X_j)$
be the sample (empirical) covariance operator based on $(X_1,\dots, X_n).$  
Denote  by $P_r$ the spectral projector of $\Sigma$ corresponding to its $r$-th eigenvalue $\mu_r$ and by $\hat P_r$ the empirical counterpart 
of $P_r.$ The main goal of the paper is to obtain tight bounds on 
$$
\sup_{x\in {\mathbb R}}
\left|{\mathbb P}\left\{\frac{\|\hat P_r-P_r\|_2^2-{\mathbb E}\|\hat P_r-P_r\|_2^2}{{\rm Var}^{1/2}(\|\hat P_r-P_r\|_2^2)}\leq x\right\}-\Phi(x)\right|,
$$
where $\|\cdot\|_2$ denotes the Hilbert--Schmidt norm and $\Phi$ is the standard normal distribution function. Such accuracy of normal approximation 
of the distribution of squared Hilbert--Schmidt error is characterized in terms of so called effective rank of $\Sigma$
defined as ${\bf r}(\Sigma)=\frac{{\rm tr}(\Sigma)}{\|\Sigma\|_{\infty}},$ where ${\rm tr}(\Sigma)$ is 
the trace of $\Sigma$ and $\|\Sigma\|_{\infty}$ is its operator norm, as well as another parameter 
characterizing the size of  ${\rm Var}(\|\hat P_r-P_r\|_2^2).$ Other results include non-asymptotic bounds and asymptotic 
representations for the mean squared Hilbert--Schmidt norm error ${\mathbb E}\|\hat P_r-P_r\|_2^2$ and the variance 
${\rm Var}(\|\hat P_r-P_r\|_2^2),$ and concentration inequalities for $\|\hat P_r-P_r\|_2^2$ around its expectation. 
\end{abstract}

\begin{keyword}[class=AMS]
\kwd[Primary ]{62H12} %\kwd[; secondary ]{60B20, 60G15}
\end{keyword}

\begin{keyword}
\kwd{Sample covariance} \kwd{Spectral projectors} \kwd{Effective rank} \kwd{Principal Component Analysis}
\kwd{Concentration inequalities} \kwd{Normal approximation} \kwd{Perturbation theory}
\end{keyword}

\end{frontmatter}

\section{Introduction}\label{Sec:Intro}

Let $X$ be a mean zero Gaussian random vector in a separable Hilbert space ${\mathbb H}$ with covariance operator 
$\Sigma={\mathbb E}(X\otimes X)$ and let $X_1,\dots, X_n$ be a sample of $n$ i.i.d. copies of $X.$ The sample covariance 
operator $\hat \Sigma=\hat \Sigma_n$ is defined as follows: $\hat \Sigma := \hat \Sigma_n:= n^{-1}\sum_{j=1}^n (X_j\otimes X_j).$ 
Denote by $\mu_r$ the $r$-th eigenvalue of $\Sigma$ (in a decreasing order) and by $P_r$ the corresponding spectral projector of $\Sigma$ (that is, the orthogonal projector on the eigenspace of eigenvalue $\mu_r$). Let $\hat P_r$ denote properly defined empirical counterpart of $P_r$ (see Section \ref{sec:pert} for a precise definition). The main goal of the 
paper is to obtain a tight bound on the accuracy of normal approximation of the distribution of the squared Hilbert--Schmidt 
norm error $\|\hat P_r-P_r\|_2^2$ of the estimator $\hat P_r.$ Another goal is to provide bounds on the risk ${\mathbb E}\|\hat P_r-P_r\|_2^2$ of this estimator as well as non-asymptotic bounds on concentration of random variables $\|\hat P_r-P_r\|_2^2$
around its expectation. These bounds will be expressed in terms of natural complexity parameters of the problem, the most 
important one being the so called {\it effective rank} ${\bf r}(\Sigma)$ that  has been recently used in the literature
(see \cite{Vershynin}, \cite{BuneaXiao}, \cite{Lounici2013}). 

\begin{definition}
The following quantity
$
{\bf r}(\Sigma):= \frac{{\rm tr}(\Sigma)}{\|\Sigma\|_{\infty}}
$
will be called the effective rank of $\Sigma.$
\end{definition}
 
Here ${\rm tr}(\Sigma)$ denotes the trace of $\Sigma$ and $\|\Sigma\|_{\infty}$ denotes its operator norm. 
The above definition clearly implies that ${\bf r}(\Sigma)\leq {\rm rank}(\Sigma).$ 
A recent result by 
Koltchinskii and Lounici, see \cite{Koltchinskii_Lounici_arxiv}, shows that, in the Gaussian case, the size of the operator norm 
error $\|\hat \Sigma-\Sigma\|_{\infty}$ of sample covariance $\hat \Sigma$ is completely characterized by $\|\Sigma\|_{\infty}$ 
and ${\bf r}(\Sigma).$ This makes the effective rank ${\bf r}(\Sigma)$ the crucial complexity parameter of the problems of  
estimation of covariance and its spectral characteristics (its principal components) that allows one to study
principal component analysis (PCA) problems in a unified dimension-free framework that includes their high-dimensional and infinite-dimensional versions  
(functional PCA, kernel PCA, etc).  As in the preceding paper \cite{Koltchinskii_Lounici_bilinear}, our goal is to study the problem 
in a ``high-complexity setting", where both the sample size $n$ and the effective rank ${\bf r}(\Sigma)$ are large, although our 
primary focus is on the case when ${\bf r}(\Sigma)=o(n)$ which implies operator norm consistency of both $\hat \Sigma$ and $\hat P_r.$ 
This setting is much closer to high-dimensional  covariance estimation and PCA problems than to standard results on PCA in Hilbert spaces with a fixed value of ${\rm tr}(\Sigma)$ (see, for instance, \cite{DPR}) that are commonly used in the literature on functional PCA and kernel PCA.  It includes, in particular, high-dimensional {\it spiked covariance models} (see \cite{Johnstone}, \cite{Johnstone_Lu}, \cite{Paul_2007}) in which 
\begin{equation}
\label{spike}
\Sigma= \sum_{j=1}^m s_j^2 (\theta_j\otimes \theta_j)+ \sigma^2 P_{p},
\end{equation}
where $\{\theta_j\}$ is an orthonormal basis of ${\mathbb H},$ $s_1^2>s_2^2>\dots >s_m^2$ 
are the variances of $m$ independent components of the ``signal", $\sigma^2$ is the variance 
of the noise components and $P_p:=\sum_{j=1}^p(\theta_j\otimes \theta_j)$ is the orthogonal projector on the linear span of the vectors $\theta_1,\dots, \theta_p,$ where $p>m.$ This models the covariance of a Gaussian signal with $m$ independent 
components observed in an independent Gaussian white noise. It is usually assumed that the number of components $m$
and the variances $s_1^2,\dots, s_m^2, \sigma^2$
are fixed, but the overall dimension of the problem $p=p_n\to \infty$ as $n\to \infty$ is large, implying that
$$
{\rm tr}(\Sigma)= \sum_{j=1}^m s_j^2 + \sigma^2 p \sim \sigma^2 p \to \infty\ {\rm as}\ n\to \infty
$$
and 
$
{\bf r}(\Sigma)\sim \frac{\sigma^2}{s_1^2+\sigma^2}p. 
$
Estimation of the components of the ``signal" $\theta_1,\dots, \theta_m$ 
is viewed as PCA for unknown covariance 
$\Sigma.$ It is common to consider a sequence of high-dimensional problems in spaces $\mathbb R^p, p=p_n$ 
(rather than explicitly embed the spaces ${\mathbb R}^p$ into an infinite dimensional Hilbert space ${\mathbb H}$).  
To assess the performance of the PCA, the loss function $L(a,b):= 2(1-|\langle a,b\rangle |),$ where $a,b \in \R^p$ are unit vectors,
was used in \cite{Birnbaumetal}.
A closely related loss function is defined by  
$L'(a,b):= \|a \otimes a - b \otimes b\|_2^2 = 2(1-\langle a,b\rangle^2),$
see, for instance, \cite{MaCaiYihong,Lounici2013,VuLei2012}. In the case of spiked covariance model with 
$\sigma^2 = 1$ and $\frac{p}{n}\rightarrow 0$ as $n\rightarrow \infty,$ 
the following asymptotic representation of the risk holds, \cite{Birnbaumetal}:
\begin{equation}
\label{Birnbaumresult}
\E L(\hat\theta_j ,\theta_j) = \left[  \frac{(p-m)(1+s_j^2)}{n s_j^4}  + \frac{1}{n}\sum_{k\neq j} \frac{(1+s_j^2)(1+s_k^2)}{(s_j^2 -s_k^2)^2}\right](1+o(1)), j=1,\dots, m.
\end{equation}
Under the assumption $\frac{p}{n}\rightarrow c>0$ as $n\rightarrow \infty$ the classical PCA is known to yield inconsistent estimators of the eigenvectors, see, e.g., \cite{Johnstone_Lu}. In \cite{Birnbaumetal}, a thresholding procedure in spirit of diagonal thresholding of Johnstone and Lu \cite{Johnstone_Lu} was proposed and it was proved that it achieves optimality in the minimax sense for the loss $L(\cdot,\cdot)$ under sparsity conditions on the eigenvectors of $\Sigma$. 

In this paper, we are not making any structural assumptions on the covariance operator $\Sigma,$ such as the spiked covariance 
model, sparsity, etc, but rather study the problem in terms of complexity parameter ${\bf r}(\Sigma).$
We derive representations of the Hilbert--Schmidt risk 
${\mathbb E}\|\hat P_r-P_r\|_2^2$ of empirical spectral projectors in the case when ${\bf r}(\Sigma)=o(n)$ that 
imply representation  (\ref{Birnbaumresult}) for spiked covariance model. Specifically, we prove 
that 
\begin{align}\label{Risk-mean}
{\mathbb E}\|\hat P_r-P_r\|_2^2=(1+o(1))\frac{A_r(\Sigma)}{n},
\end{align}
where $A_r(\Sigma)=2{\rm tr}(P_r\Sigma P_r){\rm tr}(C_r\Sigma C_r)$ and the operator 
$C_r$ is defined as $C_r:=\sum_{s\neq r}\frac{P_s}{\mu_r-\mu_s}.$ In addition, we show
that 
\begin{align}\label{Risk-var}
{\rm Var}(\|\hat P_r-P_r\|_2^2)= (1+o(1))\frac{B_r^2(\Sigma)}{n^2},
\end{align}
where $B_r(\Sigma):=2\sqrt{2}\|P_r\Sigma P_r\|_2\|C_r\Sigma C_r\|_2,$
and derive concentration bounds for random variable $\|\hat P_r-P_r\|_2^2$ around 
its expectation. One of the main results of the paper is the following bound on the accuracy
of normal approximation of random variable $\|\hat P_r-P_r\|_2^2$ that holds under rather mild assumptions:
\begin{align}
\label{normal_approx_intro}
&
\nonumber
\sup_{x\in {\mathbb R}}
\left|{\mathbb P}\left\{\frac{\|\hat P_r-P_r\|_2^2-{\mathbb E}\|\hat P_r-P_r\|_2^2}{{\rm Var}^{1/2}(\|\hat P_r-P_r\|_2^2)}\leq x\right\}-\Phi(x)\right|
\\
&
\leq 
C\left[\frac{1}{B_r(\Sigma)}+\frac{{\bf r}(\Sigma)}{B_r(\Sigma)\sqrt{n}}\sqrt{\log \left(\frac{B_r(\Sigma)\sqrt{n}}{{\bf r}(\Sigma)}\bigvee 2\right) }+ \frac{\log n}{\sqrt{n}}\right],
\end{align}
where $\Phi(x)$ denotes the standard normal distribution function. This bound implies that the distribution of random variable 
$\frac{\|\hat P_r-P_r\|_2^2-{\mathbb E}\|\hat P_r-P_r\|_2^2}{{\rm Var}^{1/2}(\|\hat P_r-P_r\|_2^2)}$ is asymptotically standard 
normal as soon as $n\to \infty,$ $B_r(\Sigma)\to \infty$ and $\frac{{\bf r}(\Sigma)}{B_r(\Sigma)\sqrt{n}}\to 0$
which, in particular, implies that ${\bf r}(\Sigma)=o(n)$. 

Throughout the paper, for $A,B>0,$ the notation $A\lesssim B$ means that there exists an absolute constant $C>0$
such that $A\leq C B.$ Similarly, $A\gtrsim B$ means that $A\geq CB$ for an absolute constant $C>0$ and $A\asymp B$
means that $A\lesssim B$ and $A\gtrsim B.$ In the cases when the constant $C$ in the above bounds might depend 
on some parameter(s), say, $\gamma,$ and we want to emphasize this dependence, we will write $A\lesssim_{\gamma} B,$ $A\gtrsim_{\gamma} B,$ or $A\asymp_{\gamma}B.$ 
Also, throughout the paper (as it was already done in the introduction), $\|\cdot\|_2$ denotes the Hilbert--Schmidt norm and $\|\cdot\|_{\infty}$ the operator norm of operators acting in ${\mathbb H}.$ 
With a minor abuse of notation, $\langle \cdot, \cdot \rangle$ denotes both the inner product of ${\mathbb H}$
and the Hilbert--Schmidt inner product.  We will also use the sign $\otimes$ to denote the tensor product. For instance,
for $u,v\in {\mathbb H},$ $u\otimes v$ is a linear operator in ${\mathbb H}$ defined as follows: $(u\otimes v)x=u \langle v,x\rangle, x\in {\mathbb H}.$

In what follows, we will frequently prove exponential bounds for certain random variables, say, $\xi,$ of the following type: for some constant $C>0$ and for all $t\geq 1,$ with probability at least $1-e^{-t},$ $\xi \leq C\sqrt{t}.$ Often, it will be proved instead that 
the inequality holds with probability, say, $1-2e^{-t}.$ In such cases, it is easy to rewrite the probability bound in the initial form 
by changing the value of the constant $C.$ For instance, replacing $t$ by $t+\log 2$ allows one to claim that with probability 
$1-e^{-t},$ $\xi \leq C\sqrt{t+\log 2}\leq C(1+\log 2)^{1/2}\sqrt{t}$ that holds for all $t\geq 1.$ In such cases, it will be said 
without further explanation that probability bound $1-e^{-2t}$ can be replaced by  $1-e^{-t}$ by adjusting the constants.

\section{Preliminaries}\label{Sec:Prelim}

In this section, we discuss recent bounds on the operator norm 
$\|\hat \Sigma_n-\Sigma\|_{\infty}$ obtained in \cite{Koltchinskii_Lounici_arxiv} and several well known results of perturbation theory used throughout the paper (see also \cite{Koltchinskii_Lounici_bilinear}). 
%and also obtain bounds on the risk of empirical spectral 
%projectors with respect to the Hilbert--Schmidt norm.  

\subsection{Bounds on the operator norm $\|\hat \Sigma_n-\Sigma\|_{\infty}.$}\label{Sec:Spectral}

In \cite{Koltchinskii_Lounici_arxiv}, it was proved that, in the Gaussian case, 
moment bounds and concentration inequalities for the operator norm $\|\hat \Sigma-\Sigma\|_{\infty}$ are completely characterized 
by the operator norm $\|\Sigma\|_{\infty}$ and the effective rank ${\bf r}(\Sigma).$ More precisely, the following theorems hold.

\begin{theorem}
\label{th_operator}
Let $X,X_1,\ldots,X_n$ be i.i.d. centered Gaussian random vectors in ${\mathbb H}$ with covariance $\Sigma = \E(X\otimes X).$ Then, for all $p\geq 1,$ 
\begin{align}
\label{E1/p}
\E^{1/p}\|\hat\Sigma- \Sigma\|_{\infty}^p \asymp_{p} \|\Sigma\|_\infty\max\left\lbrace \sqrt{\frac{\mathbf{r}(\Sigma)}{n}}, \frac{\mathbf{r}(\Sigma)}{n}\right\rbrace.
\end{align}
\end{theorem}

\begin{theorem}
\label{spectrum_sharper} 
Let $X,X_1,\ldots,X_n$ be i.i.d. centered Gaussian random vectors in ${\mathbb H}$ with  covariance $\Sigma = \E(X\otimes X).$  Then, there exist a constant 
$C>0$ such that for all $t\geq 1$ with probability at least $1-e^{-t},$ 
\begin{align}
 \label{con_con}
\Bigl|\|\hat\Sigma - \Sigma\|_{\infty}- {\mathbb E}\|\hat\Sigma - \Sigma\|_{\infty}\Bigr| \leq C\|\Sigma\|_\infty
\biggl[\biggl(\sqrt{\frac{\mathbf{r}(\Sigma)}{n}}\bigvee 1\biggr)\sqrt{\frac{t}{n}}\bigvee \frac{t}{n} 
\biggr].
\end{align}
As a consequence of this bound and (\ref{E1/p}), with some constant $C>0$ and with the same probability 
\begin{align}
\label{sha_sha}
\|\hat\Sigma - \Sigma\|_{\infty} \leq C\|\Sigma\|_\infty
\biggl[\sqrt{\frac{\mathbf{r}(\Sigma)}{n}} \bigvee \frac{\mathbf{r}(\Sigma)}{n} 
\bigvee \sqrt{\frac{t}{n}}\bigvee \frac{t}{n} 
\biggr].
\end{align}

\end{theorem}

\subsection{Perturbation theory}\label{sec:pert}

Several simple and well known facts on perturbations of linear operators    
(see Kato \cite{Kato}) will be stated in a form suitable for our purposes. The proofs 
of some of these facts that seem not to be readily available in the literature were given in 
\cite{Koltchinskii_Lounici_bilinear} (see also Koltchinskii \cite{Koltchinskii} and Kneip and Utikal 
\cite{Kneip} for some bounds in the same direction).

Let $\Sigma :{\mathbb H}\mapsto {\mathbb H}$ be a compact symmetric operator 
(in our case, the covariance operator of a random vector $X$ in ${\mathbb H}$)
with the spectrum $\sigma(\Sigma).$ The following spectral representation is well known to hold 
with the series converging in the operator norm:
$
\Sigma = \sum_{r\geq 1}\mu_r P_r,
$
where $\mu_r$ denotes distinct non-zero eigenvalues of $\Sigma$ arranged in decreasing order and 
$P_r$ the corresponding spectral projectors.  
Denote by $\sigma_i=\sigma_i(\Sigma)$ the eigenvalues of $\Sigma$ arranged in nonincreasing order and repeated with their respective multiplicities. Let $\Delta_r = \{i\,:\,\sigma_i(\Sigma) = \mu_r \}$ and let $m_r:={\rm card}(\Delta_r)$ denote the multiplicity of $\mu_r$. 
Define 
$
g_r: = g_r(\Sigma) := \mu_r-\mu_{r+1}>0, r\geq 1.
$
Let  $\bar g_r := \bar g_r(\Sigma):= \min(g_{r-1},g_r)$ for $r\geq 2$ and 
$\bar g_1:=g_1.$ The quantity $\bar g_r$ will be called {\it the $r$-th  spectral gap, or the spectral gap of eigenvalue $\mu_r$}. 

Let now $\tilde \Sigma:=\Sigma+E$ be another compact symmetric operator in ${\mathbb H}$
with spectrum $\sigma (\tilde \Sigma)$ and eigenvalues $\tilde \sigma_i=\sigma_i(\tilde \Sigma), i\geq 1$
(arranged in nonincreasing order and repeated with their multiplicities), where $E$ is a perturbation 
of $\Sigma.$ By Lidskii's inequality,
$$
\sup_{j\geq 1}|\sigma_j(\Sigma)-\sigma_j(\tilde \Sigma)|\leq \sup_{j\geq 1}|\sigma_j(E)|
=\|E\|_{\infty}.
$$
Thus, for all $r\geq 1,$
\begin{align*}
\inf_{j\not\in \Delta_r}|\tilde \sigma_j - \mu_r| \geq 
\bar g_r - \sup_{j\geq 1}|\tilde \sigma_j - \sigma_j| \geq 
\bar g_r -\|E\|_{\infty} 
\end{align*}
and 
\begin{align*}
\sup_{j\in \Delta_r}|\tilde \sigma_j - \mu_r| = \sup_{j\in \Delta_r}|\tilde \sigma_j-\sigma_j|\leq \|E\|_{\infty}. 
\end{align*}
Assuming that the perturbation $E$ is small in the sense that 

$\|E\|_{\infty}< \frac{\bar g_r}{2},$

it is easy to conclude that all the eigenvalues $\tilde \sigma_j, j\in \Delta_r$ are covered by 
an interval 
$$\Bigl(\mu_r-\|E\|_{\infty},\mu_r+\|E\|_{\infty}\Bigr)\subset (\mu_r-\bar g_r/2,\mu_r+\bar g_r/2)$$ 
and the rest of the eigenvalues of 
$\tilde \Sigma$ are outside of the interval 
$$\Bigl(\mu_r-(\bar g_r-\|E\|_{\infty}),\mu_r+(\bar g_r-\|E\|_{\infty})\Bigr)
\supset [\mu_r-\bar g_r/2,\mu_r+\bar g_r/2].$$ 
Moreover, under the assumption 
$
\|E\|_{\infty}<\frac{1}{4}\min_{1\leq s\leq r}\bar g_s=:\bar \delta_r,
$
the set $\{\sigma_j(\tilde \Sigma): j\in \bigcup_{s=1}^r\Delta_s\}$ of  the largest eigenvalues of $\tilde \Sigma$ consists of $r$ ``clusters", the diameter of each cluster being strictly smaller than $2\bar \delta_r$ and the distance between any two clusters being larger than $2\bar \delta_r.$
Thus, it is possible to identify clusters of eigenvalues of $\tilde \Sigma$ corresponding to each of the $r$ largest distinct eigenvalues
$\mu_s, s=1,\dots, r$ of $\Sigma.$ Let $\tilde P_r$ be the orthogonal projector
on the direct sum of eigenspaces of $\tilde \Sigma$ corresponding to the eigenvalues 
$\tilde \sigma_j, j\in \Delta_r$ 
(to the $r$-th cluster of eigenvalues of $\tilde \Sigma$). 
The following ``partial resolvent" operator will be frequently used throughout the paper:
$
C_r:=\sum_{s\neq r}\frac{1}{\mu_r-\mu_s}P_s.
$

We will need a couple of lemmas proved in \cite{Koltchinskii_Lounici_bilinear} (see Lemmas 1 and 4 therein):

\begin{lemma}\label{lem-pert-spectral}
%Suppose condition (\ref{bound_on_E}) holds.
The following bound holds: 
\begin{align}
\label{bd_1}
\|\tilde P_r-P_r\|_\infty \leq 4\frac{\|E\|_{\infty}}{\bar g_r}.
\end{align}
Moreover, 
\begin{equation}
\label{perture}
\tilde P_r-P_r =L_r(E)+ S_r(E),
\end{equation}
where 
\begin{align}
\label{linear_perturb}
L_r(E):=C_r E P_r + P_r E C_r
\end{align}
and 
\begin{equation}
\label{remainder_A}
\|S_r(E)\|_{\infty}\leq 14 \biggl(\frac{\|E\|_{\infty}}{\bar g_r}\biggr)^2.
\end{equation}
\end{lemma}

\begin{lemma}\label{lipSr}
Let $\gamma \in (0,1)$ and suppose that 
\begin{equation}
\label{delt_le}
\delta \leq \frac{1-\gamma}{1+\gamma}\frac{\bar g_r}{2}.
\end{equation}
Suppose also that 
\begin{equation}
\label{E_le}
\|E\|_{\infty}\leq (1+\gamma)\delta\ {\rm and}\ \|E'\|_{\infty}\leq (1+\gamma)\delta.
\end{equation}
Then, there exists a constant $C_{\gamma}>0$ such that 
\begin{equation}
\label{Srlip}
\|S_r(E)-S_r(E')\|_{\infty}\leq C_{\gamma}\frac{\delta}{\bar g_r^2}\|E-E'\|_{\infty}.
\end{equation}
\end{lemma}

\section{Bounds on the risk of empirical spectral projectors}

Let $\hat P_r$ be the orthogonal projector on the direct sum 
of eigenspaces of $\hat \Sigma$ corresponding to the eigenvalues 
$\{\sigma_j(\hat \Sigma), j\in \Delta_r\}$ (in other words, to the $r$-th 
cluster of eigenvalues of $\hat \Sigma,$ see Section \ref{sec:pert}). 

We will state simple bounds for the bias ${\mathbb E}\hat P_r-P_r$ 
and the ``variance'' ${\mathbb E}\|\hat P_r-{\mathbb E}\hat P_r\|_2^2$ that immediately 
imply a representation of the risk ${\mathbb E}\|\hat P_r-P_r\|_2^2.$

Denote 
\begin{equation}
\label{defineAn}
A_r(\Sigma):= 2{\rm tr}(P_r\Sigma P_r){\rm tr}(C_r\Sigma C_r).
\end{equation}
It is easy to see that 
\begin{equation}
\label{A_n_bou}
A_r(\Sigma) \leq 2\frac{m_r \mu_r}{\bar g_r^2} \|\Sigma\|_{\infty}{\bf r}(\Sigma)
\end{equation}
and 
\begin{equation}
\label{A_n_bou'}
A_r(\Sigma) \geq  2\left(\frac{m_r \mu_r}{\|\Sigma\|_{\infty}} {\bf r}(\Sigma)-\frac{m_r\mu_r^2}{\|\Sigma\|_{\infty}^2}\right),
\end{equation}
which implies that 
\begin{equation}
\label{A_n_bou''}
A_r(\Sigma)\asymp {\bf r}(\Sigma)
\end{equation}
(assuming that $\|\Sigma\|_{\infty}$ and $m_r$ are bounded away both from $0$ and 
from $\infty,$ $\bar g_r$ is bounded away from $0$  
and ${\bf r}(\Sigma)\to \infty$).

\begin{theorem}
\label{risk_bd}
The following bounds hold:
\begin{enumerate}
 \item 
\begin{equation}
\label{bias_A}
\|{\mathbb E}\hat P_r-P_r\|_{\infty}\lesssim \frac{\|\Sigma\|_{\infty}^2}{\bar g_r^2}
\biggl(\frac{{\bf r}(\Sigma)}{n}\bigvee \biggl(\frac{{\bf r}(\Sigma)}{n}\biggr)^2\biggr)
\end{equation}
and 
\begin{equation}
\label{bias_B}
\|{\mathbb E}\hat P_r-P_r\|_{2}\lesssim \sqrt{m_r}\frac{\|\Sigma\|_{\infty}^2}{\bar g_r^2}
\biggl(\frac{{\bf r}(\Sigma)}{n}\bigvee \biggl(\frac{{\bf r}(\Sigma)}{n}\biggr)^2\biggr). 
\end{equation}
\item
In addition,  
\begin{equation}
\label{var_var_A}
{\mathbb E}\|\hat P_r-{\mathbb E}\hat P_r\|_2^2= \frac{A_r(\Sigma)}{n} + \rho_n, 
\end{equation}
where 
\begin{equation}
\label{var_var_B}
|\rho_n| 
\leq 
\frac{m_r \|\Sigma\|_{\infty}^4}{\bar g_r^4}
\biggl(\biggl(\frac{{\bf r}(\Sigma)}{n}\biggr)^{3/2}\bigvee 
\biggl(\frac{{\bf r}(\Sigma)}{n}\biggr)^4
\biggr). 
\end{equation}
\item
If $\Sigma=\Sigma^{(n)},$ the sequences $\|\Sigma^{(n)}\|_{\infty}$ and $m_r=m_r^{(n)}$ are both bounded away from $0$ and from 
$\infty,$ $\bar g_r=\bar g_r^{(n)}$ is bounded away from $0,$ and 
$$
{\bf r}(\Sigma)=o(n),
$$ 
then the following representation holds:
\begin{equation}
\label{risk_bou}
{\mathbb E}\|\hat P_r-P_r\|_2^2= \frac{A_r(\Sigma)}{n} + O\biggl(\biggl(\frac{{\bf r}(\Sigma)}{n}\biggr)^{3/2}\biggr)= (1+o(1))\frac{A_r(\Sigma)}{n}.
\end{equation}
\end{enumerate}
\end{theorem}

\begin{rema} 
In the case of spiked covariance model (\ref{spike}) for all $r=1,\dots, m,$ 
$$
A_r(\Sigma)=2\left(\frac{(p-m)(s_r^2+\sigma^2)}{s_r^4}+ \sum_{1\leq j\leq m, j\neq r}\frac{(s_j^2+\sigma^2)(s_r^2+\sigma^2)}{(s_r^2-s_j^2)^2}\right).
$$
Assuming that $m, s_1^2,\dots, s_m^2, \sigma^2$ are fixed, $p\to \infty$ and $p=o(n)$ as $n\to \infty,$ it is easy to check that (\ref{risk_bou}) implies bound (\ref{Birnbaumresult}) obtained in \cite{Birnbaumetal}. 

\end{rema}

\begin{proof}
Recall the following relationship (see Lemma \ref{lem-pert-spectral})
\begin{equation}
\label{rep_rep_A}
\hat P_r-P_r =L_r(E)+ S_r(E),
\end{equation}
where $E:=\Hat \Sigma-\Sigma,$
$
L_r(E):=C_r E P_r + P_r E C_r
$
and
$
S_r(E) := \hat P_r-P_r-L_r(E).
%- \frac{1}{2\pi i} \oint_{\gamma_r}\sum_{k\geq 2} (-1)^{k} [R_{\Sigma}(\eta)E]^k %R_{\Sigma}(\eta)d\eta.
$
Clearly, $C_rP_r=P_rC_r=0$ (due to the orthogonality of $P_r$ and $P_s, s\neq r$). Also,  $P_r X$ and $C_r X$ are independent random variables (since, by the same orthogonality property, they are uncorrelated and $X$ is Gaussian). 

To prove Claim 1, note that, 
since ${\mathbb E}L_r(E)=0,$ we have  
$
{\mathbb E}\hat P_r-P_r={\mathbb E}S_r(E).
$
Therefore, by bound (\ref{remainder_A}) of Lemma \ref{lem-pert-spectral}, we get 
\begin{equation}
\label{bou_sre}
\|{\mathbb E}\hat P_r-P_r\|_{\infty}\leq 
{\mathbb E}\|S_r(E)\|_{\infty}\leq 14 \frac{{\mathbb E}\|E\|_{\infty}^2}{\bar g_r^2}.
\end{equation}
Bound (\ref{bias_A}) now follows from Theorem \ref{th_operator}.  
Bound (\ref{bias_B}) is also obvious since $\hat P_r, P_r$ are 
operators of rank $m_r,$ $L_r(E)$ is of rank at most $2m_r$ and 
$S_r(E)=\hat P_r-P_r-L_r(E)$ is of rank at most $4m_r.$ Thus,
$\|S_r(E)\|_2\lessim \sqrt{m_r}\|S_r(E)\|_{\infty},$ and 
the result follows from the previous bounds.

To prove Claim 2, note that 
$
\hat P_r-{\mathbb E}\hat P_r=L_r(E)+S_r(E)-{\mathbb E}S_r(E).
$
Therefore,
\begin{equation}
\label{gu_gu}
\|\hat P_r-{\mathbb E}\hat P_r\|_2^2=\|L_r(E)\|_2^2+\|S_r(E)-{\mathbb E}S_r(E)\|_2^2
+2\Bigl\langle L_r(E),S_r(E)-{\mathbb E}S_r(E)\Bigr\rangle. 
\end{equation}
The following representations are obvious: 
\begin{equation}
\label{CEP}
C_r E P_r = n^{-1}\sum_{j=1}^n C_r X_j \otimes P_r X_j,\ \ 
P_r E C_r = n^{-1}\sum_{j=1}^n P_r X_j \otimes C_r X_j. 
\end{equation}
Note that, by (\ref{CEP}), due to orthogonality of $C_r E P_r, P_r E C_r$ and due to independence of $P_r X, C_r X,$
\begin{align}
\label{ga_ga}
&
\nonumber
{\mathbb E}\|L_r(E)\|_2^2= {\mathbb E}\|C_r EP_r + P_r E C_r\|_2^2
= {\mathbb E}\Bigl(\|C_r E P_r\|_2^2 + \|P_r E C_r\|_2^2\Bigr)=
2 {\mathbb E}\|C_r E P_r\|_2^2 
\\
&
\nonumber
=2{\mathbb E}\biggl\|n^{-1}\sum_{j=1}^n P_rX_j\otimes C_r X_j\biggr\|_2^2
=\frac{2{\mathbb E}\|P_r X\otimes C_r X\|_2^2}{n} 
= \frac{2{\mathbb E}\|P_r X\|^2 \|C_r X\|^2}{n}
\\
&
=\frac{2{\mathbb E}\|P_r X\|^2 {\mathbb E}\|C_r X\|^2}{n}= 
\frac{2{\rm tr}(P_r\Sigma P_r){\rm tr}(C_r\Sigma C_r)}{n}=\frac{A_r(\Sigma)}{n}.
\end{align}

Next, note that 
$
{\mathbb E}\|S_r(E)-{\mathbb E}S_r(E)\|_2^2\leq {\mathbb E}\|S_r(E)\|_2^2.
$
Recall that $S_r(E)$ is of rank $\leq 4m_r$ and $\|S_r(E)\|_2^2\leq 4m_r\|S_r(E)\|_{\infty}^2.$ Quite similarly to (\ref{bou_sre}), one 
can prove that 
$
{\mathbb E}\|S_r(E)\|_{\infty}^2\lesssim \frac{1}{\bar g_r^4} {\mathbb E}\|E\|_{\infty}^4.
$
Therefore, by Theorem \ref{th_operator}, we get 
\begin{equation}
\label{go_go}
{\mathbb E}\|S_r(E)-{\mathbb E}S_r(E)\|_2^2\lesssim m_r\frac{\|\Sigma\|_{\infty}^4}{\bar g_r^4} \biggl(\biggl(\frac{{\bf r}(\Sigma)}{n}\biggr)^2\bigvee 
\biggl(\frac{{\bf r}(\Sigma)}{n}\biggr)^4
\biggr).  
\end{equation}
As a consequence of (\ref{ga_ga}) and (\ref{go_go}), it easily follows that 
\begin{align}
\label{ge_ge}
&
{\mathbb E}\Bigl|\Bigl\langle L_r(E),S_r(E)-{\mathbb E}S_r(E)\Bigr\rangle\Bigr|
\leq {\mathbb E}^{1/2}\|L_r(E)\|_2^2 {\mathbb E}^{1/2}\|S_r(E)-{\mathbb E}S_r(E)\|_2^2
\\
&
\nonumber
\lessim \sqrt{\frac{A_r(\Sigma)}{n}} \sqrt{m_r}\frac{\|\Sigma\|_{\infty}^2}{\bar g_r^2} \biggl(\frac{{\bf r}(\Sigma)}{n}\bigvee 
\biggl(\frac{{\bf r}(\Sigma)}{n}\biggr)^2\biggr)
\\
&
\nonumber
\lessim m_r\frac{\|\Sigma\|_{\infty}^3}{\bar g_r^3} \biggl(\biggl(\frac{{\bf r}(\Sigma)}{n}\biggr)^{3/2}\bigvee 
\biggl(\frac{{\bf r}(\Sigma)}{n}\biggr)^{5/2}\biggr)
\end{align}
(\ref{var_var_A}) and (\ref{var_var_B}) now follow from 
(\ref{gu_gu}), (\ref{ga_ga}), (\ref{go_go}) and (\ref{ge_ge}). 

Claim 3 is an easy consequence of the first two claims due to the ``bias-variance 
decomposition'' 
$
{\mathbb E}\|\hat P_r-P_r\|_2^2 = \|{\mathbb E}\hat P_r-P_r\|_2^2 +{\mathbb E}\|\hat P_r-{\mathbb E}\hat P_r\|_2^2 
$ 
(see also (\ref{A_n_bou''})).
\qed
\end{proof}

\section{Concentration Inequalities}

The main goal of this section is to derive a concentration bound for the squared Hilbert--Schmidt error $\|\hat P_r-P_r\|_2^2$
around its expectation. 
Denote 
\begin{equation}
\label{define_B_r}
B_r(\Sigma):=2\sqrt{2}\|P_r\Sigma P_r\|_2 \|C_r\Sigma C_r\|_2.
\end{equation}

\begin{theorem}
\label{HS-conc}
Suppose that, for some $\gamma \in (0,1),$ 
\begin{equation}
\label{cond_gamma}
{\mathbb E}\|\hat \Sigma-\Sigma\|_{\infty} \leq \frac{(1-\gamma)\bar g_r}{2}.
\end{equation}
Moreover, let $t\geq 1$ and suppose that   
\begin{equation}
\label{siml_assump}
m_r\lesssim 1,\ \ \frac{\|\Sigma\|_{\infty}}{\bar g_r}\sqrt{\frac{t}{n}}\lesssim 1.
\end{equation}
Then, for some constant $D_{\gamma}>0$ with probability at least $1-e^{-t},$
\begin{equation}   
\label{HS-CONC}
\Bigl| \|\hat P_r-P_r\|_2^2- {\mathbb E}\|\hat P_r-P_r\|_2^2\Bigr|\leq
D_{\gamma}\biggl[\frac{B_r(\Sigma)}{n}\sqrt{t} \bigvee \frac{\|\Sigma\|_{\infty}^2}{\bar g_r^2}\frac{t}{n}\bigvee 
\frac{\|\Sigma\|_{\infty}^3}{\bar g_r^3}\frac{{\bf r}(\Sigma)}{n}\sqrt{\frac{t}{n}}\biggr].
\end{equation} 
\end{theorem}

Note that the first term $\frac{B_r(\Sigma)}{n}\sqrt{t}$ in the right hand side of (\ref{HS-CONC}) is 
dominant if $B_r(\Sigma)\to \infty$ and $\frac{{\bf r}(\Sigma)}{B_r(\Sigma) \sqrt{n}}\to 0.$ In the 
next section, it will be shown that under the same assumptions the random variable  
$\frac{\|\hat P_r-P_r\|_2^2- {\mathbb E}\|\hat P_r-P_r\|_2^2}{{\rm Var}^{1/2}(\|\hat P_r-P_r\|_2^2)}$
is close in distribution to the standard normal and, in addition, 
${\rm Var}^{1/2}(\|\hat P_r-P_r\|_2^2)=(1+o(1))\frac{B_r(\Sigma)}{n}.$

The main ingredient in the proofs of these results is a concentration bounds for the random variables $\|\hat P_r-P_r\|_2^2 - \|L_r(E)\|_2^2$ given below. 

\begin{theorem}
\label{technical_2}

Suppose that, for some $\gamma \in (0,1),$ condition (\ref{cond_gamma}) holds. 

Then, there exists a constant $L_{\gamma}>0$ such that for all $t\geq 1$
the following bound holds with probability at least $1-e^{-t}:$ 
\begin{align}
\label{remaind+}
&
\Bigl|\|\hat P_r-P_r\|_2^2 - \|L_r(E)\|_2^2-{\mathbb E}(\|\hat P_r-P_r\|_2^2 - \|L_r(E)\|_2^2)\Bigr|
\\
&
\nonumber
\leq L_\gamma m_r\frac{\|\Sigma\|_{\infty}^3}{\bar g_r^3}
\biggl(\frac{{\bf r}(\Sigma)}{n}\bigvee \frac{t}{n}\bigvee \biggl(\frac{t}{n}\biggr)^2\biggr)
\sqrt{\frac{t}{n}}.
\end{align}
\end{theorem}

\begin{proof} 
It easily follows from Theorem \ref{th_operator} that under assumption (\ref{cond_gamma})
$$
\|\Sigma\|_{\infty}\biggl(\sqrt{\frac{{\bf r}(\Sigma)}{n}}\bigvee\frac{{\bf r}(\Sigma)}{n}\biggr)
\lesssim \frac{(1-\gamma)\bar g_r}{2}\leq \|\Sigma\|_{\infty},
$$
which implies that ${\bf r}(\Sigma)\lesssim n.$ Theorem \ref{spectrum_sharper} implies 
that for some constant $C'>0$ and for all $t\geq 1$ with probability at least $1-e^{-t}$
$$
\|\hat \Sigma-\Sigma\|_{\infty}\leq {\mathbb E}\|\hat \Sigma-\Sigma\|_{\infty}+ C'\|\Sigma\|_{\infty}\biggl(\sqrt{\frac{t}{n}}\bigvee \frac{t}{n}\biggr).
$$
We will first assume that 
\begin{equation}
\label{case_odin}
C\|\Sigma\|_{\infty}\sqrt{\frac{t}{n}}\leq \frac{\gamma \bar g_r}{4}
\end{equation}
with a sufficiently large constant $C\geq 1$ (the proof of the concentration bound in the opposite case will be much easier). 
This assumption easily implies that $t\leq n$ and, if $C\geq C',$ 
$$
C'\|\Sigma\|_{\infty}\biggl(\sqrt{\frac{t}{n}}\bigvee \frac{t}{n}\biggr)\leq C\|\Sigma\|_{\infty}\sqrt{\frac{t}{n}}.
$$
Denote 
$$
\delta_n(t):={\mathbb E}\|\hat \Sigma-\Sigma\|_{\infty}+C\|\Sigma\|_{\infty}\sqrt{\frac{t}{n}}.
$$
Then ${\mathbb P}\{\|\hat \Sigma-\Sigma\|_{\infty}\geq \delta_n(t)\}\leq e^{-t}.$

As before, denote $E=\hat \Sigma-\Sigma.$
The main part of the proof is the derivation of a concentration inequality for the function
$$
g(X_1,\dots,X_n)= 
\biggl(\|\hat P_r-P_r\|_2^2 - \|L_r(E)\|_2^2\biggr)\varphi \biggl(\frac{\|E\|_{\infty}}{\delta}\biggr),
$$
where, for some $\gamma \in (0,1),$ $\varphi$ is a Lipschitz function on ${\mathbb R}_{+}$ with constant $\frac{1}{\gamma},$ $0\leq \varphi (s)\leq 1,$
$\varphi (s)=1, s\leq 1,$ $\varphi(s)=0, s>1+\gamma,$
and $\delta>0$ is such that $\|E\|_{\infty}\leq \delta$
with a high probability. This inequality will be then used with $\delta=\delta_n(t).$ Together with 
Theorem \ref{spectrum_sharper}, it will imply bound (\ref{remaind+}) under the assumption (\ref{case_odin}).

Our main tool is 
the following concentration inequality that easily follows 
from Gaussian isoperimetric inequality.

\begin{lemma}
\label{Gaussian_concentration}
Let $X_1,\dots, X_n$ be i.i.d. centered Gaussian random variables 
in ${\mathbb H}$ with covariance operator $\Sigma.$
Let $f:{\mathbb H}^n\mapsto {\mathbb R}$ be a function satisfying 
the following Lipschitz condition with some $L>0:$
$$
\Bigl|f(x_1,\dots, x_n)-f(x_1',\dots, x_n')\Bigr|\leq 
L\biggl(\sum_{j=1}^n \|x_j-x_j'\|^2\biggr)^{1/2},\ x_1,\dots, x_n,x_1',\dots,x_n'\in {\mathbb H}.
$$ 
Suppose that, for a real number $M,$
$$
{\mathbb P}\{f(X_1,\dots,X_n)\geq M\}\geq 1/4\ {\rm and}\ {\mathbb P}\{f(X_1,\dots,X_n)\leq M\}\geq 1/4.
$$
Then, there exists a numerical constant $D>0$ such that for all $t\geq 1,$ 
$$
{\mathbb P}\Bigl\{|f(X_1,\dots, X_n)-M|\geq DL\|\Sigma\|_{\infty}^{1/2}\sqrt{t}\Bigr\}\leq 
e^{-t}.
$$
\end{lemma}

We have to check now that the function $g(X_1,\dots,X_n)$ satisfies the Lipschitz condition
(with a minor abuse of notation we view $X_1,\dots, X_n$ here as non-random vectors in ${\mathbb H}$
rather than random variables). 

\begin{lemma}
\label{Lipschitz_constant_2}
Suppose that, for some $\gamma\in (0,1/2),$ 
\begin{equation}
\label{con_delta''}
\delta \leq \frac{1-2\gamma}{1+2\gamma}\frac{\bar g_r}{2}.
\end{equation}
Then, there exists a numerical constant $D_{\gamma}>0$ such that, for all $X_1,\dots, X_n,
X_1',\dots, X_n'\in {\mathbb H},$
\begin{equation}
\label{lip_lip_CC'}
|g(X_1,\dots, X_n)-g(X_1',\dots, X_n')| 
\leq 
D_{\gamma} m_r\frac{\delta^2}{\bar g_r^3}
\frac{\|\Sigma\|_{\infty}^{1/2}+\sqrt{\delta}}{\sqrt{n}}
\biggl(\sum_{j=1}^n \|X_j-X_j'\|^2\biggr)^{1/2}.
\end{equation}
\end{lemma}

\begin{proof} Observe that  
\begin{align*}
&
\|\hat P_r-P_r\|_2^2-\|L_r(E)\|_2^2 = \|L_r(E)+S_r(E)\|_2^2 -
\|L_r(E)\|_2^2= 
\\
&
2\Bigl\langle L_r(E),S_r(E)\Bigr\rangle + \|S_r(E)\|_2^2=:\tilde g(E).
\end{align*}
Also, note that $L_r(E)$ is an operator 
of rank at most $2m_r$ and $S_r(E)=\hat P_r-P_r-L_r(E)$ has rank at most 
$4m_r$ (under the assumption that $\|E\|_{\infty}<\bar g_r/2$ implying that 
$\hat P_r$ is of rank $m_r$). This allows us to bound the Hilbert--Schmidt 
norms of such operators in terms of their operator norms: $\|A\|_2^2\leq {\rm rank}(A)\|A\|_{\infty}^2.$ Thus, we get
$$
\left| g(X_1,\ldots,X_n) \right| \leq 4\sqrt{2} m_ r\left(\|L_r(E)\|_{\infty}\|S_r(E)\|_{\infty} + \|S_r(E)\|_\infty^2\right)\varphi\left( \frac{\|E\|_{\infty}}{\delta}\right).
$$
Since $\varphi \biggl(\frac{\|E\|_{\infty}}{\delta}\biggr)=0$
if $\|E\|_{\infty}\geq (1+\gamma)\delta,$ claims 
(\ref{linear_perturb}), (\ref{remainder_A}) of Lemma \ref{lem-pert-spectral} imply that, under assumption (\ref{con_delta''})
\begin{align}\label{bdf}
|g(X_1,\dots, X_n)|\leq  c_\gamma  m_r  \biggl(\frac{\delta}{\bar g_r}\biggr)^3,
\end{align}
for some constant $c_\gamma>0$ depending only $\gamma$.

We will denote $\hat \Sigma':=n^{-1}\sum_{j=1}^n X_j'\otimes X_j'$ and $E':=\hat \Sigma'-\Sigma.$
Using now (\ref{linear_perturb}), (\ref{remainder_A}), (\ref{bdf}) and the fact that $\varphi$ is bounded by $1$ and Lipschitz with 
constant $\frac{1}{\gamma},$ which implies that the function $t\mapsto \varphi\biggl(\frac{t}{\delta}\biggr)$ is Lipschitz with constant $\frac{1}{\gamma \delta},$ we easily get that, under the assumptions
\begin{equation}
\label{assumEE}
\|E\|_{\infty}\leq (1+\gamma)\delta,\ \|E'\|_{\infty}\leq (1+\gamma)\delta, 
\end{equation}
the following inequality holds:
\begin{align}\label{bdEE'}
&
\biggl|\tilde g(E) \varphi \biggl(\frac{\|E\|_{\infty}}{\delta}\biggr)-
\tilde g(E') \varphi \biggl(\frac{\|E'\|_{\infty}}{\delta}\biggr)\biggr|
\\
&
\nonumber
\leq |\tilde g(E)-\tilde g(E')|+\frac{c_{\gamma}}{\gamma}
\frac{\delta^2}{\bar g_r^3}\|E-E'\|_{\infty}
\\
&
\nonumber
\leq 2 |\bigl \langle L_r(E-E') , S_r(E)   \bigr \rangle| + 2 |\bigl \langle L_r(E'), S_r(E )- S_r(E') \bigr\rangle | \\
&
\nonumber
\hspace{3cm}+ |\bigl\langle S_r(E)- S_r(E'),S_r(E)+S_r(E') \bigr \rangle|+  \frac{c_{\gamma}}{\gamma}
\frac{\delta^2}{\bar g_r^3}\|E-E'\|_{\infty}.
\end{align}

Using the Lipschitz bound of Lemma \ref{lipSr} and (\ref{linear_perturb}), (\ref{remainder_A}) of Lemma \ref{lem-pert-spectral},

we easily get that
\begin{align}\label{bdassEE}
\biggl| g(X_1,\ldots,X_n) -
 g(X_1',\ldots,X_n') \biggr|
 \leq c_\gamma' m_r\frac{\delta^2}{\bar g_r^3} \|E - E'\|_{\infty},
\end{align}
where $c_{\gamma}'>0$ depends only on $\gamma$.

A similar bound holds in the case when 
$$
\|E\|_{\infty}\leq (1+\gamma)\delta,\ \|E'\|_{\infty}> (1+\gamma)\delta
$$ 
(when both norms are larger than $(1+\gamma)\delta,$ the function $\varphi$ is equal to zero and the bound is trivial). 
Indeed, first consider the case when $\|E-E'\|_{\infty}\geq \gamma \delta.$ Then, in view of (\ref{bdf}), we have 
\begin{align*}
&
\biggl|\tilde g(E) \varphi \biggl(\frac{\|E\|_{\infty}}{\delta}\biggr)-
\tilde g(E') \varphi \biggl(\frac{\|E'\|_{\infty}}{\delta}\biggr)\biggr|
\\
&
=\biggl|\tilde g(E)  \varphi\biggl(\frac{\|E\|_{\infty}}{\delta}\biggr)\biggr|\leq 
c_\gamma m_r\frac{\delta^3}{\bar g_r^3} \leq 
\frac{c_\gamma}{\gamma} m_r\frac{\delta^2}{\bar g_r^3}\|E-E'\|_{\infty}.
\end{align*}
On the other hand, if $\|E-E'\|_{\infty}<\gamma \delta,$ we have that $\|E'\|_{\infty}\leq  (1+2\gamma)\delta$ and, taking into account assumption (\ref{con_delta''}), we can repeat the argument in the case (\ref{assumEE}) ending up with the same bound as (\ref{bdassEE}) with a positive constant (possibly different from $c_{\gamma}',$ but still depending only on $\gamma$) in the right hand side.

The following bound (see Lemma 5 in \cite{Koltchinskii_Lounici_bilinear}) provides a control of $\|E - E'\|_{\infty}:$
\begin{equation}
\label{E__E'}
\|E-E'\|_{\infty}\leq 
\frac{4\|\Sigma\|_{\infty}^{1/2}+4\sqrt{2\delta}}{\sqrt{n}}
\biggl(\sum_{j=1}^n \|X_j-X_j'\|^2\biggr)^{1/2}
\bigvee
\frac{4}{n}
\sum_{j=1}^n \|X_j-X_j'\|^2.
\end{equation}
Now substitute the last bound in the right hand side of (\ref{bdassEE}) and  
observe that, in view of (\ref{bdf}), the left hand side of (\ref{bdassEE}) can be 
also upper bounded by 
$
2c_\gamma m_r\frac{\delta^3}{\bar g_r^3}.
$
Therefore, we get that with some constant $L_{\gamma}>0,$
\begin{align}
\label{bdEE''}
&
\biggl| g(X_1,\ldots,X_n) -
 g(X_1',\ldots,X_n') \biggr|
\\
&
\nonumber
\leq 4c_{\gamma}'m_r\frac{\delta^2}{\bar g_r^3}
\biggl[
\frac{\|\Sigma\|_{\infty}^{1/2}+\sqrt{2\delta}}{\sqrt{n}}
\biggl(\sum_{j=1}^n \|X_j-X_j'\|^2\biggr)^{1/2}
\bigvee
\frac{1}{n}
\sum_{j=1}^n \|X_j-X_j'\|^2
\biggr]
\bigwedge 
2c_\gamma m_r\frac{\delta^3}{\bar g_r^3}
\\
&
\nonumber
\leq L_{\gamma}m_r\frac{\delta^2}{\bar g_r^3}\biggl[
\frac{\|\Sigma\|_{\infty}^{1/2}+\sqrt{2\delta}}{\sqrt{n}}
\biggl(\sum_{j=1}^n \|X_j-X_j'\|^2\biggr)^{1/2}
\bigvee
\biggl(\frac{1}{n}
\sum_{j=1}^n \|X_j-X_j'\|^2\bigwedge \delta 
\biggr)\biggr].
\end{align}
Using an elementary inequality $a\wedge b\leq \sqrt{ab}, a,b\geq 0,$
we get 
$$
\frac{1}{n}
\sum_{j=1}^n \|X_j-X_j'\|^2\bigwedge \delta \leq 
\sqrt{\frac{\delta}{n}}\biggl(\sum_{j=1}^n \|X_j-X_j'\|^2\biggr)^{1/2}.
$$
This allows us to drop the last term in the maximum in the right hand side of 
(\ref{bdEE''}) (since a similar expression is a part of the first term). 
This yields bound (\ref{lip_lip_CC'}).

\qed

\end{proof}

Getting back to the proof of Theorem \ref{technical_2}, it will be convenient to prove first a version of its concentration bound with a median instead 
of the mean. Denote by ${\rm Med}(\eta)$ a median of a random 
variable $\eta$ and define $M:={\rm Med}\Bigl(\|\hat P_r-P_r\|_2^2 - \|L_r(E)\|_2^2\Bigr).$ Let $\delta:=\delta_n(t)$
and suppose that $t\geq \log(4)$ (by adjusting the constants, one can replace this condition by $t\geq 1$ as it is done 
in the statement of the theorem). Under conditions (\ref{cond_gamma}) and (\ref{case_odin}), 
$\delta_n(t)\leq \Bigl(1-\frac{\gamma}{2}\Bigr)\frac{\bar g_r}{2}=\frac{1-2\gamma'}{1+2\gamma'}\frac{\bar g_r}{2}$
for some $\gamma'\in (0,1/2).$ Thus, the function $g(X_1,\dots, X_n)$ satisfies the Lipschitz condition (\ref{lip_lip_CC'})
with some constant $D_{\gamma}'=D_{\gamma'}.$ Also, we have ${\mathbb P}\{\|E\|_{\infty}\geq \delta\}\leq e^{-t}\leq 1/4.$  Note that 
on the event $\{\|E\|_{\infty}<\delta\},$ $g(X_1,\dots, X_n)=\|\hat P_r-P_r\|_2^2 - \|L_r(E)\|_2^2.$ Therefore,
\begin{align*}
&
{\mathbb P}\{g(X_1,\dots, X_n)\geq M\}\geq 
{\mathbb P}\{g(X_1,\dots, X_n)\geq M, \|E\|_{\infty}<\delta\}\geq 
\\
&
{\mathbb P}\Bigl\{\|\hat P_r-P_r\|_2^2 - \|L_r(E)\|_2^2\geq M\Bigr\}-{\mathbb P}\{\|E\|_{\infty}\geq \delta\}\geq 1/4. 
\end{align*}
Quite similarly,
$
{\mathbb P}\{g(X_1,\dots, X_n)\leq M\}\geq 1/4.
$
It follows from Lemma \ref{Gaussian_concentration} 
that with probability at least $1-e^{-t}$
$$
\Bigl|g(X_1,\dots, X_n)-M\Bigr| \leq 
L_{\gamma}' m_r\frac{\delta_n(t)^2}{\bar g_r^3}
\|\Sigma\|_{\infty}^{1/2}\Bigl(\|\Sigma\|_{\infty}^{1/2}+\sqrt{\delta_n(t)}\Bigr)\sqrt{\frac{t}{n}}
$$
with some constant $L_{\gamma}'>0.$
Using the bound 
$$
\delta_n(t)\lessim \|\Sigma\|_{\infty}\biggl(\sqrt{\frac{{\bf r}(\Sigma)}{n}}\bigvee \sqrt{\frac{t}{n}}\biggr) 
$$
that easily follows from the definition of $\delta_n(t)$ and the bound of Theorem \ref{th_operator},
we get that with some $L_{\gamma}>0$ and with the same probability
$$
\Bigl|g(X_1,\dots, X_n)-M\Bigr| \leq 
L_\gamma m_r\frac{\|\Sigma\|_{\infty}^3}{\bar g_r^3}
\biggl(\frac{{\bf r}(\Sigma)}{n}\bigvee \frac{t}{n}\biggr)
\sqrt{\frac{t}{n}}.
$$
Since ${\mathbb P}\{\|E\|_{\infty}\geq \delta\}\leq e^{-t}$ and $g(X_1,\dots, X_n)=\|\hat P_r-P_r\|_2^2 - \|L_r(E)\|_2^2$
when $\|E\|_{\infty}<\delta,$ we can conclude that with probability at least $1-2e^{-t}$
\begin{align*}
&
\Bigl|\|\hat P_r-P_r\|_2^2 - \|L_r(E)\|_2^2-M\Bigr| \leq 
L_\gamma m_r\frac{\|\Sigma\|_{\infty}^3}{\bar g_r^3}
\biggl(\frac{{\bf r}(\Sigma)}{n}\bigvee \frac{t}{n}\biggr)
\sqrt{\frac{t}{n}}
\\
&
\leq L_\gamma m_r\frac{\|\Sigma\|_{\infty}^3}{\bar g_r^3}
\biggl(\frac{{\bf r}(\Sigma)}{n}\bigvee \frac{t}{n}\bigvee
\biggl(\frac{t}{n}\biggr)^2\biggr)
\sqrt{\frac{t}{n}}
\end{align*}
Adjusting the value of the constant $L_{\gamma}$ one can replace the probability bound $1-2e^{-t}$ by 
$1-e^{-t}.$

We will now prove a similar bound in the case when condition (\ref{case_odin}) does not hold. Then,  
\begin{equation}
\label{case_dva}
\frac{\|\Sigma\|_{\infty}}{\bar g_r}\sqrt{\frac{t}{n}}\geq \frac{\gamma }{4C}.
\end{equation}
It follows from bound (\ref{bd_1}) and the definition of $L_r(E)$ that, for some constant $c>0,$ 
$$
\Bigl|\|\hat P_r-P_r\|_2^2 - \|L_r(E)\|_2^2\Bigr|\leq 
c m_r\frac{\|E\|_{\infty}^2}{\bar g_r^2}.
$$
We can now use the bounds of theorems \ref{th_operator} and \ref{spectrum_sharper} to show that under 
condition (\ref{cond_gamma}) for some $C>0$ with probability at least $1-e^{-t}$
$$
\Bigl|\|\hat P_r-P_r\|_2^2 - \|L_r(E)\|_2^2\Bigr|\leq 
C m_r \frac{\|\Sigma\|_{\infty}^2}{\bar g_r^2} \biggl(\frac{{\bf r}(\Sigma)}{n}\bigvee \frac{t}{n}\bigvee \biggl(\frac{t}{n}\biggr)^2\biggr).
$$
In view of condition (\ref{case_dva}), we get from the last bound that with some $L_{\gamma}'>0$
with probability at least $1-e^{-t}$
$$
\Bigl|\|\hat P_r-P_r\|_2^2 - \|L_r(E)\|_2^2\Bigr| \leq 
L_\gamma' m_r\frac{\|\Sigma\|_{\infty}^3}{\bar g_r^3}
\biggl(\frac{{\bf r}(\Sigma)}{n}\bigvee \frac{t}{n}\bigvee
\biggl(\frac{t}{n}\biggr)^2\biggr)
\sqrt{\frac{t}{n}}.
$$
This easily implies the following bound on the median $M:$ 
$$
M\leq 
L_\gamma' m_r\frac{\|\Sigma\|_{\infty}^3}{\bar g_r^3}
\biggl(\frac{{\bf r}(\Sigma)}{n}\bigvee \frac{\log 2}{n}\bigvee
\biggl(\frac{\log 2}{n}\biggr)^2\biggr)
\sqrt{\frac{\log 2}{n}}.
$$
Therefore, for some $L_{\gamma}>0$ and for all $t\geq 1,$
with probability at least $1-e^{-t}$
\begin{equation}
\label{exp_med}
\Bigl|\|\hat P_r-P_r\|_2^2 - \|L_r(E)\|_2^2-M\Bigr| \leq 
L_\gamma m_r\frac{\|\Sigma\|_{\infty}^3}{\bar g_r^3}
\biggl(\frac{{\bf r}(\Sigma)}{n}\bigvee \frac{t}{n}\bigvee
\biggl(\frac{t}{n}\biggr)^2\biggr)
\sqrt{\frac{t}{n}},
\end{equation}
and the last bound was proved in both cases (\ref{case_odin}) and (\ref{case_dva}). 

It remains to integrate out the tails of exponential bound (\ref{exp_med}) to get the inequality 
\begin{align*}
&
\Bigl|{\mathbb E}(\|\hat P_r-P_r\|_2^2 - \|L_r(E)\|_2^2)-M\Bigr| \leq 
{\mathbb E}\Bigl|\|\hat P_r-P_r\|_2^2 - \|L_r(E)\|_2^2-M\Bigr| \leq
\\
&
\bar L_\gamma m_r\frac{\|\Sigma\|_{\infty}^3}{\bar g_r^3}
\biggl(\frac{{\bf r}(\Sigma)}{n}\bigvee \frac{1}{n}\biggr)
\sqrt{\frac{1}{n}}
\end{align*}
with some $\bar L_{\gamma}>0,$ which, along with (\ref{exp_med}), implies concentration inequality (\ref{remaind+}).

\qed

\end{proof}

We now turn to the proof of Theorem \ref{HS-conc}.

\begin{proof}
In view of Theorem \ref{technical_2}, it is sufficient to obtain a concentration bound for 
$
\|L_r(E)\|_2^2 - \mathbb E \|L_r(E)\|_2^2.
$
This could be done by rewriting $\|L_r(E)\|_2^2$ in terms of $U$-statistics and using the 
corresponding exponential bounds. However, we will follow a different (more elementary) path that directly 
utilizes the Gaussiness of random variables $\{X_j\}.$ The key ingredient is the following 
simple representation lemma. In what follows, $\xi\stackrel{d}{=}\eta$ means that random variables 
$\xi$ and $\eta$ have the same distribution. 

\begin{lemma}
\label{repr_d}
The following representation holds:
\begin{align}
\label{V-con-interm-1}
n\|L_r(E)\|_2^2 \stackrel{d}{=} 2\sum_{k\in \Delta_r} \gamma_k \|C_r X^{(k)}\|^2,
\end{align}
where $\gamma_k$ are the eigenvalues of the random matrix $\Gamma_r := \frac{1}{n} \sum_{i=1}^n P_r X_i \otimes P_r X_i$ and $X^{(k)}$, $k\in \Delta_r$ are i.i.d. copies of $X$ independent of $\Gamma_r.$
\end{lemma}

\begin{proof}
Note that 
$
n\|L_r(E)\|_2^2 = 
n\|P_r E C_r+ C_r E P_r\|_2^2.
$
Since the operators $P_r E C_r$ and $C_r E P_r$ are orthogonal with respect to 
the Hilbert--Schmidt inner product and 
$$
\|P_r E C_r\|_2^2 = {\rm tr}(P_r E C_r C_r E P_r)= {\rm tr}(C_r E P_r P_r E C_r)=
\|C_r E P_r\|_2^2,
$$
we have 
$$
\|P_r E C_r+ C_r E P_r\|_2^2= \|P_r E C_r\|_2^2+\|C_r E P_r\|_2^2= 
2\|P_r E C_r\|_2^2.
$$
Also, note that $P_rEC_r=\frac{1}{n}\sum_{j=1}^n P_r X_j\otimes C_r X_j.$
Therefore,
\begin{equation}
\label{repr_LL''}
n\|L_r(E)\|_2^2 = 2n\|P_r E C_r\|_2^2=
2 \biggl\|\frac{1}{\sqrt{n}}\sum_{j=1}^n P_r X_j\otimes C_r X_j\biggr\|_2^2.
\end{equation}

Define the following mapping 
$$
T(u_1\otimes u_2\otimes u_3\otimes u_4)=(u_1\otimes u_3\otimes u_2\otimes u_4), u_1,u_2,u_3,u_4\in {\mathbb H}.
$$
It can be extended in a unique way by linearity and continuity to a bounded linear operator 
$
T:
{\mathbb H}\otimes {\mathbb H}\otimes {\mathbb H}\otimes {\mathbb H}\mapsto {\mathbb H}\otimes {\mathbb H}\otimes {\mathbb H}\otimes {\mathbb H}.
$

Recall that $P_r X_j, j=1,\dots, n$ and $C_r X_j, j=1,\dots, n$ are centered Gaussian random variables and they are uncorrelated
(see the proof of Theorem \ref{risk_bd}). Therefore, they are also independent. 
Conditionally on $P_r X_j, j=1,\dots, n,$ the distribution 
of random operator 
$
U:=\frac{1}{\sqrt{n}}\sum_{j=1}^n P_r X_j\otimes C_r X_j
$
is centered Gaussian with covariance 
\begin{align*}
&
{\mathbb E}(U\otimes U|P_r X_j, j=1,\dots, n)=n^{-1}\sum_{j=1}^n 
{\mathbb E}\Bigl(P_r X_j \otimes C_r X_j\otimes P_r X_j\otimes C_r X_j\big\vert P_r X_j, j=1,\dots, n\Bigr)
\\
&
=T(\Gamma_r\otimes {\mathbb E}(C_r X\otimes C_r X))=
T(\Gamma_r \otimes (C_r\Sigma C_r)).
\end{align*}
%where 
%$$
%\Gamma_r:=n^{-1}\sum_{j=1}^n P_r X_j \otimes P_r X_j.
%$$
Note that $\Gamma_r$ can be viewed as a symmetric operator 
acting in the eigenspace of eigenvalue $\mu_r,$ and it is 
nonnegatively definite. Thus, it has spectral representation 
$
\Gamma_r = \sum_{k\in \Delta_r}\gamma_k \phi_k\otimes \phi_k,
$ 
where $\gamma_k\geq 0$ are its eigenvalues and $\phi_k$ are its orthonormal eigenvectors (that belong to the eigenspace 
of $\mu_r$). It follows that  
\begin{align*}
&
{\mathbb E}(U\otimes U|P_r X_j, j=1,\dots, n)
=T\left(\sum_{k\in \Delta_r}\gamma_k \Bigl(\phi_k\otimes \phi_k\otimes {\mathbb E}(C_r X\otimes C_r X)\Bigr)\right).
\end{align*}
Let $X^{(k)}, k\in \Delta_r$ be independent copies of $X$
(also independent of $X_1,\dots, X_n$). Denote 
$
V:= \sum_{k\in \Delta_r} \sqrt{\gamma_k} 
\phi_k \otimes C_rX^{(k)}.
$ 
It is now easy to check that 
\begin{align*}
&
{\mathbb E}(V\otimes V|P_r X_j, j=1,\dots, n)
=T\left( \sum_{k\in \Delta_r}\gamma_k \Bigl(\phi_k\otimes \phi_k\otimes {\mathbb E}(C_r X\otimes C_r X)\Bigr)\right),
\end{align*}
implying that conditional distributions of $U$ and $V$ given 
$P_r X_j, j=1,\dots, n$ are the same. As a consequence, 
the distribution of $n\|L_r(E)\|_2^2=2\|U\|_2^2$ coincides 
with the distribution of random variable 
\begin{align}
\label{repr_LL'''}
2\|V\|_2^2= 
2\sum_{k\in \Delta_r} \gamma_k \Bigl\|\phi_k \otimes C_rX^{(k)}\Bigr\|_2^2= 
2\sum_{k\in \Delta_r} \gamma_k \|C_rX^{(k)}\|^2.
\end{align}
\qed
\end{proof}

Note that  
\begin{align*}
\|C_r X^{(k)}\|^2=
\sum_{s\neq r} \sum_{j \in \Delta_s} \frac{\mu_s}{(\mu_s - \mu_r)^2}\eta_{k,j}^2,
\end{align*}
where 
$
\eta_{k,j}:= \mu_s^{-1/2}\langle X^{(k)},\theta_j\rangle, j\in \Delta_s, k\in \Delta_r, s\neq r 
$ 
are i.i.d. standard normal random variables, $\{\theta_j: j\in \Delta_s\}$ being an orthonormal basis 
of the eigenspace corresponding to $\mu_s, s\geq 1.$ 
In view of representation (\ref{repr_LL'''}), we get
$$
n\|L_r(E)\|_2^2\stackrel{d}{=}
2\sum_{k\in \Delta_r}\sum_{s\neq r} \sum_{j \in \Delta_s} \frac{\gamma_k \mu_s}{(\mu_s - \mu_r)^2}\eta_{k,j}^2.
$$
and, since $\gamma_k, k\in \Delta_r$ and $\eta_{k,j}, j\in \Delta_s, k\in \Delta_r$ are independent,
\begin{align*}
&
n{\mathbb E}\|L_r(E)\|_2^2=
2\sum_{k\in \Delta_r}\sum_{s\neq r} \sum_{j \in \Delta_s} \frac{{\mathbb E}\gamma_k \mu_s}{(\mu_s - \mu_r)^2}
=2\sum_{s\neq r} \frac{{\mathbb E}{\rm tr}(\Gamma_r) m_s\mu_s}{(\mu_s - \mu_r)^2}
\\
&
=
2\sum_{s\neq r} \frac{{\rm tr}(P_r\Sigma P_r)m_s\mu_s}{(\mu_s - \mu_r)^2}
=2\sum_{s\neq r} \frac{m_r \mu_r m_s \mu_s}{(\mu_s - \mu_r)^2}
=2{\rm tr}(P_r\Sigma P_r){\rm tr}(C_r\Sigma C_r)=A_r(\Sigma).
\end{align*}
Therefore,
\begin{align}
\label{relay_a}
&
\|L_r(E)\|_2^2 -{\mathbb E}\|L_r(E)\|_2^2 \stackrel{d}{=}
\frac{2}{n}\sum_{k\in \Delta_r}\sum_{s\neq r} \sum_{j \in \Delta_s} \frac{\mu_r \mu_s}{(\mu_s - \mu_r)^2}\frac{\gamma_k}{\mu_r}(\eta_{k,j}^2-1)
\\
&
\nonumber
+
\frac{2}{n}\sum_{k\in \Delta_r}\sum_{s\neq r} \frac{\mu_r m_s \mu_s}{(\mu_s - \mu_r)^2}\biggl(\frac{\gamma_k}{\mu_r}-1\biggr).
\end{align}

In order to control the right hand side in the above display, the following elementary lemma will be used.

\begin{lemma}
\label{lem:chi-dev}
Let $\{\xi_k\}$ be i.i.d. standard normal random variables. 
There exists a numerical constant $c>0$ such that for all $t>0$
$$
\mathbb P \left(\left|\sum_{k}\lambda_k (\xi_k^2 - 1)\right| \geq t \right) \leq 2\left( \exp\left\lbrace  - \frac{c t^2}{\sum_k \lambda_k^2}\right\rbrace \bigvee \exp\left\lbrace - \frac{c t}{\sup_k |\lambda_k|} \right\rbrace\right).
$$
\end{lemma}

\begin{proof}
By a simple computation,
$$
\mathbb E \exp\left\lbrace  u \sum_k \lambda_k (\xi_k^2 - 1)  \right\rbrace = \prod_k \frac{1}{\sqrt{e^{2u \lambda_k}(1 - 2 u \lambda_k)}}
$$
for all $u>0$ such that $2u \sup_k \lambda_k < 1.$
Since $e^{x}(1 - x) \geq (1+x) (1- x) = 1 - x^2$ for $x\in (0,1)$, we easily get
$$
\frac{1}{\sqrt{e^x(1-x)}} \leq \frac{1}{\sqrt{1-x^2}} \leq \sqrt{1+2x^2} \leq e^{x^2}, x\in (0, 2^{-1/2}).
$$
This implies that for all $u>0$ satisfying the condition 
$
2u \sup_k \lambda_k < 2^{-1/2},
$
the following bound holds:
$$
\mathbb E \exp\left\lbrace  u \sum_k \lambda_k (\xi_k^2 - 1)  \right\rbrace \leq \exp\left \lbrace 4 u^2 \sum_k \lambda_k^ 2\right\rbrace
$$
The bound on $\mathbb P \left(\sum_{k}\lambda_k (\xi_k^2 - 1) \geq t \right)$
now follows by a standard application of Markov's inequality and optimizing the resulting bound with respect to $u.$

Similarly, 
$$
\mathbb E \exp\left\lbrace  u \sum_k \lambda_k (1-\xi_k^2 )  \right\rbrace = \prod_k \frac{e^{u\lambda_k}}{\sqrt{1+2 u \lambda_k}}.
$$
Since $\frac{e^x}{\sqrt{1+2x}}=\exp\left\{x-\frac{1}{2}\log (1+2x)\right\}\leq e^{x^2}, x\geq 0,$ we 
get 
$$
\mathbb E \exp\left\lbrace  u \sum_k \lambda_k (1-\xi_k^2 )  \right\rbrace \leq \exp \biggl\{u^2\sum_{k}\lambda_k^2\biggr\}, u\geq 0,
$$
implying  the bound on the lower tail. 
\qed
\end{proof}

Applying the bound of the lemma to the first term in the right hand side of relationship (\ref{relay_a}) 
{\it conditionally on $\gamma_k, k\in \Delta_r,$} 
we get that with probability at least $1-e^{-t}$
\begin{align*}
&
\biggl|\frac{2}{n}\sum_{k\in \Delta_r}\sum_{s\neq r} \sum_{j \in \Delta_s} \frac{\mu_r \mu_s}{(\mu_s - \mu_r)^2}\frac{\gamma_k}{\mu_r}(\eta_{k,j}^2-1)\biggr|\lesssim 
\\
&
\biggl(\sum_{s\neq r} \frac{\mu_r^2 m_s \mu_s^2}{(\mu_s - \mu_r)^4}\frac{\sum_{k\in \Delta_r}\gamma_k^2}{\mu_r^2}\biggr)^{1/2}
\frac{\sqrt{t}}{n}\bigvee \sup_{s\neq r} \frac{\mu_r \mu_s}{(\mu_s - \mu_r)^2}\frac{\sup_{k\in \Delta_r}\gamma_k}{\mu_r}\frac{t}{n}.
\end{align*}
Since $\sup_{k\in \Delta_r}\gamma_k=\|\Gamma_r\|_{\infty},$ $\sum_{k\in \Delta_r}\gamma_k^2\leq m_r \|\Gamma_r\|_{\infty}^2$
and 
$$
B_r^2(\Sigma)=8\sum_{s\neq r} \frac{m_r \mu_r^2 m_s \mu_s^2}{(\mu_s-\mu_r)^4},
$$
the last bound can be rewritten as 
\begin{align}
\label{term_1}
&
\biggl|\frac{2}{n}\sum_{k\in \Delta_r}\sum_{s\neq r} \sum_{j \in \Delta_s} \frac{\mu_r \mu_s}{(\mu_s - \mu_r)^2}\frac{\gamma_k}{\mu_r}(\eta_{k,j}^2-1)\biggr|\lesssim 
%\\
%&
B_r(\Sigma) \frac{\|\Gamma_r\|_{\infty}}{\mu_r}\frac{\sqrt{t}}{n}
\bigvee \frac{\|\Sigma\|_{\infty}^2}{\bar g_r^2}\frac{\|\Gamma_r\|_{\infty}}{\mu_r}\frac{t}{n}.
\end{align}
As to the second term in the right hand of (\ref{relay_a}), the following bound is straightforward:
\begin{align}
\label{term_2}
&
\biggl|\frac{2}{n}\sum_{k\in \Delta_r}\sum_{s\neq r} \frac{\mu_r m_s \mu_s}{(\mu_s - \mu_r)^2}\biggl(\frac{\gamma_k}{\mu_r}-1\biggr)\biggr|
\\
&
\nonumber
\leq \frac{2}{n} \sum_{s\neq r}\frac{m_r \mu_r m_s\mu_s}{(\mu_s - \mu_r)^2}\frac{\|\Gamma_r - P_r \Sigma P_r\|_{\infty}}{\mu_r}
=\frac{A_r(\Sigma)}{n}\frac{\|\Gamma_r - P_r \Sigma P_r\|_{\infty}}{\mu_r}. 
\end{align}
Theorems \ref{th_operator} and \ref{spectrum_sharper} easily imply that for all $t\geq 1$ with probability at least $1-e^{-t}$
\begin{equation}
\nonumber
\|\Gamma_r-\mu_r P_r\|_{\infty}=\biggl\|n^{-1}\sum_{j=1}^n P_r X_j \otimes P_r X_j - {\mathbb E}(P_r X\otimes P_r X)\biggr\|_{\infty}
\lesssim \mu_r \biggl(\sqrt{\frac{m_r}{n}}\bigvee \frac{m_r}{n}\bigvee \sqrt{\frac{t}{n}}\bigvee \frac{t}{n}\biggr).
\end{equation}
Under additional assumptions $m_r\lesssim n,$ $t\lesssim n,$ this bound could be simplified as 
\begin{equation}
\label{term_3}
\|\Gamma_r-\mu_r P_r\|_{\infty}
\lesssim \mu_r \biggl(\sqrt{\frac{m_r}{n}}\bigvee \sqrt{\frac{t}{n}}\biggr)
\end{equation}
and it implies that $\frac{\|\Gamma_r\|_{\infty}}{\mu_r}\lesssim 1. $ 

Thus, representation (\ref{relay_a}) and bounds (\ref{term_1}), (\ref{term_2}) imply 
that with probability at least $1-e^{-t}$
\begin{equation}
\label{L_r_conc}
\left|\|L_r(E)\|_2^2 -{\mathbb E}\|L_r(E)\|_2^2\right|\lesssim 
B_r(\Sigma) \frac{\sqrt{t}}{n}
\bigvee \frac{\|\Sigma\|_{\infty}^2}{\bar g_r^2}\frac{t}{n}
\bigvee \frac{A_r(\Sigma)}{n}\left(\sqrt{\frac{m_r}{n}}\bigvee \sqrt{\frac{t}{n}}\right).
\end{equation}
 
To complete the proof, it is enough to combine bound (\ref{L_r_conc}) with concentration inequality 
of Theorem \ref{technical_2}, to use bound  (\ref{A_n_bou}) to control $A_r(\Sigma)$ and to take into 
account conditions (\ref{siml_assump}) to simplify the resulting bound.

\qed

\end{proof}

\section{Normal approximation of squared Hilbert--Schmidt norm errors of empirical spectral projectors}
\label{Sec:CLT-HS}

The main result of this section is the following theorem:

\begin{theorem}
\label{th:normal_approx}
Suppose that, for some constants $c_1,c_2>0,$ $m_r\leq c_1$ and $\|\Sigma\|_{\infty}\leq c_2 \bar g_r.$ 
Suppose also condition (\ref{cond_gamma}) holds with some $\gamma\in (0,1).$
Then, the following bounds hold with some constant $C>0$ depending only on $\gamma, c_1, c_2:$
\begin{align}
\label{normal_approx_A}
&
\nonumber
\sup_{x\in {\mathbb R}}
\left|{\mathbb P}\left\{\frac{n}{B_r(\Sigma)}\left(\|\hat P_r-P_r\|_2^2-{\mathbb E}\|\hat P_r-P_r\|_2^2\right)\leq x\right\}-\Phi(x)\right|
\\
&
\leq 
C\left[\frac{1}{B_r(\Sigma)}+\frac{{\bf r}(\Sigma)}{B_r(\Sigma)\sqrt{n}}\sqrt{\log \left(\frac{B_r(\Sigma)\sqrt{n}}{{\bf r}(\Sigma)}\bigvee 2\right) }+ \frac{\log n}{\sqrt{n}}\right]
\end{align}
and
\begin{align}
\label{normal_approx}
&
\nonumber
\sup_{x\in {\mathbb R}}
\left|{\mathbb P}\left\{\frac{\|\hat P_r-P_r\|_2^2-{\mathbb E}\|\hat P_r-P_r\|_2^2}{{\rm Var}^{1/2}(\|\hat P_r-P_r\|_2^2)}\leq x\right\}-\Phi(x)\right|
\\
&
\leq 
C\left[\frac{1}{B_r(\Sigma)}+\frac{{\bf r}(\Sigma)}{B_r(\Sigma)\sqrt{n}}\sqrt{\log \left(\frac{B_r(\Sigma)\sqrt{n}}{{\bf r}(\Sigma)}\bigvee 2\right) }+ \frac{\log n}{\sqrt{n}}\right],
\end{align}
where $\Phi(x)$ denotes the distribution function of standard normal random variable.
\end{theorem}

This result essentially means that as soon as $B_r(\Sigma)\to \infty$ and $\frac{{\bf r}(\Sigma)}{B_r(\Sigma)\sqrt{n}}\to 0$
as $n\to \infty$ (for $\Sigma=\Sigma^{(n)}$), 
the sequence of random variables $\frac{\|\hat P_r-P_r\|_2^2-{\mathbb E}\|\hat P_r-P_r\|_2^2}{{\rm Var}^{1/2}(\|\hat P_r-P_r\|_2^2)}$
is asymptotically standard normal.

We will first establish the following fact that would allow us to replace $\frac{B_r(\Sigma)}{n}$
in bound (\ref{normal_approx_A}) by a normalizing factor ${\rm Var}^{1/2}(\|\hat P_r-P_r\|_2^2)$
in bound (\ref{normal_approx}).

\begin{theorem}
\label{th:varvar}
Suppose condition (\ref{cond_gamma}) holds for some $\gamma\in (0,1).$ 
Then the following bound holds with some constant $C_{\gamma}>0:$
\begin{align}
\label{variance_bd}
&
%\nonumber
\left|\frac{n}{B_r(\Sigma)}{\rm Var}^{1/2}(\|\hat P_r-P_r\|_2^2)-1\right|
\leq 
%\\
%&
C_{\gamma} m_r\frac{\|\Sigma\|_{\infty}^3}{\bar g_r^3}\frac{{\bf r}(\Sigma)}{B_r(\Sigma)\sqrt{n}}+
\frac{m_r+1}{n}.
\end{align}
\end{theorem}

Bound (\ref{variance_bd}) shows that, under the assumptions 
$m_r\frac{\|\Sigma\|_{\infty}^3}{\bar g_r^3}\frac{{\bf r}(\Sigma)}{B_r(\Sigma)\sqrt{n}}=o(1)$ and $m_r=o(n),$
we have 
$${\rm Var}^{1/2}(\|\hat P_r-P_r\|_2^2)=(1+o(1))\frac{B_r(\Sigma)}{n}.$$

\begin{rema}
Note that in the case of spiked covariance model (\ref{spike}), for $r=1,\dots, m,$ 
\begin{equation}
\label{define_B_r_spike}
B_r(\Sigma):=2\sqrt{2}\biggl(\sum_{1\leq j\leq m, j\neq r}\frac{(s_r^2+\sigma^2)(s_j^2+\sigma^2)}{(s_r^2-s_j^2)^4} 
+\frac{(s_r^2+\sigma^2)^2 \sigma^4 (p-m)}{s_r^8}\biggr)^{1/2},
\end{equation}
which, under the assumption that the parameters $m, s_1^2,\dots, s_m^2, \sigma^2$ are fixed, but $p=p_n\to \infty$ as $n\to \infty$
yields that 
\begin{equation}
\label{B_r_large_p}
B_r(\Sigma)=(1+o(1))\frac{2\sqrt{2}(s_r^2+\sigma^2)\sigma^2\sqrt{p}}{s_r^4}\ \ {\rm as}\ \ n\to\infty.
\end{equation}
Note also that ${\bf r}(\Sigma)\sim \frac{\sigma^2}{s_1^2+\sigma^2}p.$ Thus, the condition $p=o(n)$ implies 
$\frac{{\bf r}(\Sigma)}{B_r(\Sigma)\sqrt{n}}\to 0$ as $n\to\infty.$ 
Therefore, Theorem \ref{th:varvar} yields that 
$$
{\rm Var}^{1/2}(\|\hat P_r-P_r\|_2^2)=(1+o(1))\frac{2\sqrt{2}(s_r^2+\sigma^2)\sigma^2}{s_r^4}\frac{\sqrt{p}}{n}.
$$
Moreover, the bounds on the accuracy of normal approximation 
of Theorem \ref{th:normal_approx} are of the order 
$$
\sup_{x\in {\mathbb R}}
\left|{\mathbb P}\left\{\frac{\|\hat P_r-P_r\|_2^2-{\mathbb E}\|\hat P_r-P_r\|_2^2}{{\rm Var}^{1/2}(\|\hat P_r-P_r\|_2^2)}\leq x\right\}-\Phi(x)\right|
\leq C\biggl[\frac{1}{\sqrt{p}}+\sqrt{\frac{p}{n}\log \left(\frac{n}{p}\vee 2\right)} + \frac{\log n}{\sqrt{n}}\biggr],
$$
so, the asymptotic normality of $\|\hat P_r-P_r\|_2^2$ holds if $p=p_n\to \infty$ and $p=o(n)$ as $n\to \infty.$
\end{rema}

\begin{proof}
%We first prove that
%\begin{align}\label{var-LE}
%\mathrm{Var}(\|L(E^{(n)})\|_2^2) = \left( \frac{B_{n}}{n}\right)^2 (1+ o(1)).
%\end{align}
In view of relationships $n\|L_r(E)\|_2^2\stackrel{d}{=}2\|V\|_2^2$ and (\ref{repr_LL'''}) (see the proof of Lemma \ref{repr_d}), we have
\begin{align}
\label{var-LE-1}
\mathrm{Var}(\|L_r(E)\|_2^2) &= \frac{4}{n^2} \mathrm{Var}(\|V\|_2^2)= \frac{4}{n^2} \mathrm{Var}\left(\sum_{k\in \Delta_r} \gamma_k \|C_rX^{(k)}\|^2\right)\notag\\
&= \frac{4}{n^2} \mathbb E\left[\mathrm{Var}\left(\sum_{k\in \Delta_r} \gamma_k \|C_rX^{(k)}\|^2 \Big| P_rX_1,\ldots,P_r X_n\right) \right]\notag\\
&\hspace{1cm} + \frac{4}{n^2}\mathrm{Var}\left( \mathbb E \left[ \sum_{k\in \Delta_r} \gamma_k \|C_rX^{(k)}\|^2 \Big| P_rX_1,\ldots,P_r X_n \right]\right).
\end{align}
Recall that $\gamma_k$, $k \in \Delta_r$ depend only $P_rX_1,\ldots,P_r X_n$ and that $X^{(k)}$, $k\in \Delta_r$ are independent of $X_1,\ldots, X_n$. Thus, we get
\begin{align}
\label{var-LE-2}
&\mathbb E\left[\mathrm{Var}\left(\sum_{k\in \Delta_r} \gamma_k \|C_rX^{(k)}\|^2 \Big| P_rX_1,\ldots,P_r X_n\right) \right]=\notag\\
% &\hspace{2cm}= \mathbb E\left[\sum_{k\in \Delta_r} \gamma_k^2 \mathrm{Var}\left( \|C_rX^{(k)}\|^2 \Big| P_rX_1,\ldots,P_r X_n\right) \right]\notag\\
&=\E\left[\sum_{k\in \Delta_r} \gamma_k^2 \mathrm{Var}\left( \|C_rX^{(k)}\|^2 \right) \right]= \sum_{k\in \Delta_r} \E\left[\gamma_k^2\right] \mathrm{Var}\left( \|C_rX^{(k)}\|^2 \right)\notag\\
&=   \E\left[\|\Gamma_r\|_2^2\right] \mathrm{Var}\left( \|C_rX\|^2 \right).
%= \frac{B_{n}^2}{4} \left( 1 + O\left(\frac{m_r}{n} \right)\right).
%&\hspace{2cm}= \left(\sum_{k\in \Delta_r} E\left[\gamma_k\right] \right)\mathrm{Var}\left( \|C_rX^{(1)}\|^2 \right).
\end{align}
By an easy computation,
$$
{\mathbb E}\|\Gamma_r\|_2^2 = {\mathbb E}\left\|n^{-1}\sum_{j=1}^n P_r X_j \otimes P_r X_j\right\|_2^2
=m_r \mu_r^2 \left(1+\frac{m_r+1}{n}\right)
$$
and, for i.i.d. standard normal random variables $\{\eta_j\}$ 
$$
{\rm Var}(\|C_r X\|^2)= {\rm Var}\left(\sum_{s\neq r}\sum_{j\in \Delta_s}\frac{\mu_s}{(\mu_s-\mu_r)^2}\eta_j^2\right)=
\frac{1}{4m_r \mu_r^2} B_r^2(\Sigma). 
$$
Therefore,
\begin{equation}
\label{var-LE-10}
\mathbb E\left[\mathrm{Var}\left(\sum_{k\in \Delta_r} \gamma_k \|C_rX^{(k)}\|^2 \Big| P_rX_1,\ldots,P_r X_n\right)\right]
= \frac{B_r^2(\Sigma)}{4}\left(1+\frac{m_r+1}{n}\right).
\end{equation}
Similarly, we have 
\begin{align}
\label{var-LE-3}
&\mathrm{Var}\left( \mathbb E \left[ \sum_{k\in \Delta_r} \gamma_k \|C_rX^{(k)}\|^2 \Big| P_rX_1,\ldots,P_r X_n \right]\right)= \mathrm{Var}\left( \sum_{k\in \Delta_r} \gamma_k \mathbb E \left[ \|C_rX^{(k)}\|^2 \right]\right) \notag\\
&\hspace{2cm} = \mathrm{Var}\left( \mathrm{tr}(\Gamma_r)\right) \left(\mathbb E \left[ \|C_rX\|^2 \right]\right)^2
%\notag\\
% &\hspace{2cm}= \frac{\mu_r A_{n}^2}{2 n } = o\left(B_{n}^{2}\right).
\end{align}
and 
$$
\mathrm{Var}\left( \mathrm{tr}(\Gamma_r)\right)
=\frac{2m_r \mu_r^2}{n},\ \ \ 
\left(\mathbb E \left[ \|C_rX\|^2 \right]\right)^2=\frac{1}{4 m_r^2 \mu_r^2}A_r^2(\Sigma),
$$
implying that 
\begin{equation}
\label{var-LE-11}
\mathrm{Var}\left( \mathbb E \left[ \sum_{k\in \Delta_r} \gamma_k \|C_rX^{(k)}\|^2 \Big| P_rX_1,\ldots,P_r X_n \right]\right)
=
\frac{A_r^2(\Sigma)}{2m_r n}. 
\end{equation}
It follows from (\ref{var-LE-1}), (\ref{var-LE-10}) and (\ref{var-LE-11}) that 
\begin{equation}
\label{var-LE-11''}
\mathrm{Var}(\|L_r(E)\|_2^2)= 
\frac{B_r^2(\Sigma)}{n^2}\left(1+\frac{m_r+1}{n}\right)
+\frac{2A_r^2(\Sigma)}{m_r n^3}.
\end{equation}

Denote now
\begin{equation}
\label{xi_eta}
\xi := \|\hat P_r-P_r\|_2^2-{\mathbb E}\|\hat P_r-P_r\|_2^2,\quad \eta:= \|L_r(E)\|_2^2-{\mathbb E}\|L_r(E)\|_2^2,
\end{equation}
and $\sigma_\xi^2 = \mathbb E\xi^2$, $\sigma_\eta^2 = \mathbb E\eta^2$. 
Combining concentration bound of Theorem \ref{technical_2} with the identity $\mathbb E |\xi - \eta|^{2} = 
\int_{0}^{\infty} \mathbb P\{|\xi - \eta|^2>t\} dt,$ we obtain that
\begin{align}
\label{sig-eta-1}
|\sigma_{\xi} - \sigma_{\eta}| &\leq \sqrt{\mathbb E |\xi - \eta|^2}\leq C_{\gamma}  m_r \frac{\|\Sigma\|_{\infty}^3}{\bar g_r^3} \frac{\mathbf{r}(\Sigma)}{n}  \frac{1}{\sqrt{n}},
\end{align}
for some $C_{\gamma}>0$ depending only on $\gamma$. 

To complete the proof, observe that identity (\ref{var-LE-11''}) implies that 
\begin{align*}
&
\left|\frac{n}{B_r(\Sigma)}{\rm Var}^{1/2}(\|L_r(E)\|_2^2)-1\right| 
\leq \sqrt{1+\frac{m_r+1}{n}+\frac{2A_r^2(\Sigma)}{m_r B_r^2(\Sigma) n}}-1
\\
&
\leq 
\sqrt{1+\frac{m_r+1}{n}}-1 + \frac{\sqrt{2}A_r(\Sigma)}{\sqrt{m_r} B_r(\Sigma) \sqrt{n}}
\leq \frac{m_r+1}{n} + \frac{\sqrt{2}A_r(\Sigma)}{\sqrt{m_r} B_r(\Sigma) \sqrt{n}},
\end{align*}
then bound $A_r(\Sigma)$ using (\ref{A_n_bou}) and combine the resulting 
bound with (\ref{sig-eta-1}).

\qed

\end{proof}

We now return to the proof of Theorem \ref{th:normal_approx}.

\begin{proof}
Under notations (\ref{xi_eta}), we will upper bound 
$
\sup_{x\in {\mathbb R}}\left|{\mathbb P}\left\{\frac{n}{B_r(\Sigma)}\xi\leq x\right\}-\Phi(x)\right|.
$
Theorem \ref{th:varvar} will allow us to rewrite the normalizing factor in terms of the variance.
First recall that by Theorem \ref{technical_2}, with probability at least $1-e^{-t},$
\begin{equation}
\label{conc-conc}
|\xi-\eta|\leq 
L_\gamma m_r\frac{\|\Sigma\|_{\infty}^3}{\bar g_r^3}
\biggl(\frac{{\bf r}(\Sigma)}{n}\bigvee \frac{t}{n}\bigvee \biggl(\frac{t}{n}\biggr)^2\biggr)
\sqrt{\frac{t}{n}}.
\end{equation}
Also, by (\ref{relay_a}),
\begin{align}
\label{relay_a''}
&
\eta \stackrel{d}{=}
\frac{2}{n}\sum_{k\in \Delta_r}\sum_{s\neq r} \sum_{j \in \Delta_s} \frac{\mu_r \mu_s}{(\mu_s - \mu_r)^2}(\eta_{k,j}^2-1)
\\
&
\nonumber
+\frac{2}{n}\sum_{k\in \Delta_r}\sum_{s\neq r} \sum_{j \in \Delta_s} \frac{\mu_r \mu_s}{(\mu_s - \mu_r)^2}
\left(\frac{\gamma_k}{\mu_r}-1\right)(\eta_{k,j}^2-1)
\\
&
\nonumber
+
\frac{2}{n}\sum_{k\in \Delta_r}\sum_{s\neq r} \frac{\mu_r m_s \mu_s}{(\mu_s - \mu_r)^2}\biggl(\frac{\gamma_k}{\mu_r}-1\biggr)
\\
&
\nonumber 
=: \zeta_1+\zeta_2+\zeta_3.
\end{align}
Similarly to bound (\ref{term_1}), we get that with probability at least $1-e^{-t}$
\begin{align}
\label{term_1'}
&
|\zeta_2|\lesssim 
%\\
%&
B_r(\Sigma) \frac{\|\Gamma_r-\mu_r P_r\|_{\infty}}{\mu_r}\frac{\sqrt{t}}{n}
\bigvee \frac{\|\Sigma\|_{\infty}^2}{\bar g_r^2}\frac{\|\Gamma_r-\mu_r P_r\|_{\infty}}{\mu_r}\frac{t}{n}.
\end{align}
Assume that $1\leq t\lesssim n$ and $m_r\lesssim n.$
It follows from (\ref{term_1'}), (\ref{term_2}), (\ref{term_3}) and also from bound (\ref{A_n_bou}) on 
$A_r(\Sigma)$ that 
\begin{align}
\label{zeta_23}
&
\left|\frac{n}{B_r(\Sigma)}\left(\zeta_2+\zeta_3\right)\right|
\\
&
\nonumber
\lesssim \left(\sqrt{\frac{m_r}{n}}\bigvee \sqrt{\frac{t}{n}}\right)\sqrt{t}\bigvee 
\frac{\|\Sigma\|_{\infty}^2}{\bar g_r^2}\left(\sqrt{\frac{m_r}{n}}\bigvee \sqrt{\frac{t}{n}}\right)
\frac{t}{B_r(\Sigma)} \bigvee \frac{A_r(\Sigma)}{B_r(\Sigma)}\left(\sqrt{\frac{m_r}{n}}\bigvee \sqrt{\frac{t}{n}}\right)
\\
&
\nonumber
\lesssim
\left(\sqrt{\frac{m_r}{n}}\bigvee \sqrt{\frac{t}{n}}\right)\sqrt{t}\bigvee 
\frac{\|\Sigma\|_{\infty}^2}{\bar g_r^2}\left(\sqrt{\frac{m_r}{n}}\bigvee \sqrt{\frac{t}{n}}\right)
\frac{t}{B_r(\Sigma)} \bigvee 
m_r \frac{\|\Sigma\|_{\infty}^2}{\bar g_r^2}\frac{{\bf r}(\Sigma)}{B_r(\Sigma)}\left(\sqrt{\frac{m_r}{n}}\bigvee \sqrt{\frac{t}{n}}\right).
\end{align}
Under the assumptions of the theorem $m_r\lesssim 1,$ $\|\Sigma\|_{\infty}\lesssim \bar g_r,$ it is easy to get from (\ref{conc-conc}), (\ref{relay_a''}) and (\ref{zeta_23}) that 
\begin{equation}
\label{rep_xi}
\frac{n}{B_r(\Sigma)}\xi \stackrel{d}{=}\tau + \zeta,
\end{equation}
where
\begin{equation}
\label{xi_zeta}
\tau :=
\frac{2}{B_r(\Sigma)}\sum_{k\in \Delta_r}\sum_{s\neq r} \sum_{j \in \Delta_s} \frac{\mu_r \mu_s}{(\mu_s - \mu_r)^2}(\eta_{k,j}^2-1)
\end{equation}
and the remainder $\zeta$ satisfies the following bound with probability at least $1-e^{-t}:$
\begin{equation}
\label{zeta_zeta}
|\zeta|\lesssim \frac{t}{\sqrt{n}}\bigvee \frac{{\bf r}(\Sigma)}{B_r(\Sigma)\sqrt{n}}\sqrt{t}\bigvee \frac{t^{3/2}}{B_r(\Sigma)\sqrt{n}}. 
\end{equation}

We now use Berry-Esseen Theorem and a simple limiting argument that allows one to apply it to 
a (possibly) infinite sum of independent random variables (\ref{xi_zeta}) to get the following bound:
\begin{align}
\label{Berry}
&
\sup_{x\in {\mathbb R}} \left|{\mathbb P}\left\{\tau \leq x\right\}-\Phi(x)\right|
%\\
%&
%\nonumber
\lesssim \frac{\sum_{s\neq r} \frac{m_r \mu_r^3 m_s \mu_s^3}{(\mu_s - \mu_r)^6}}{\left(\sum_{s\neq r} \frac{m_r \mu_r^2 m_s \mu_s^2}{(\mu_s - \mu_r)^4}\right)^{3/2}}
\lesssim \frac{\|\Sigma\|_{\infty}^2}{\bar g_r^2}\frac{1}{B_r(\Sigma)}, 
\end{align}
where we also used the fact that $B_r^2(\Sigma)=8\sum_{s\neq r} \frac{m_r \mu_r^2 m_s \mu_s^2}{(\mu_s - \mu_r)^4}.$

It follows from (\ref{rep_xi}), (\ref{zeta_zeta}) and (\ref{Berry}) that with some constants $c', c''>0,$ for all $x\in {\mathbb R},$
\begin{align}
\label{B-Ess-1}
&
{\mathbb P}\left\{\frac{n}{B_r(\Sigma)}\xi \leq x\right\}
\leq 
{\mathbb P}\left\{\tau \leq x+c'\left(\frac{t}{\sqrt{n}}\bigvee \frac{{\bf r}(\Sigma)}{B_r(\Sigma)\sqrt{n}}\sqrt{t}\bigvee \frac{t^{3/2}}{B_r(\Sigma)\sqrt{n}}\right)\right\}+ e^{-t} 
\\
&
\nonumber
\leq 
\Phi\left(x+c'\left(\frac{t}{\sqrt{n}}\bigvee \frac{{\bf r}(\Sigma)}{B_r(\Sigma)\sqrt{n}}\sqrt{t}\bigvee \frac{t^{3/2}}{B_r(\Sigma)\sqrt{n}}\right)\right)+ e^{-t}+ \frac{c''}{B_r(\Sigma)}
\\
& 
\nonumber
\leq 
\Phi(x)+
c'\left(\frac{t}{\sqrt{n}}\bigvee \frac{{\bf r}(\Sigma)}{B_r(\Sigma)\sqrt{n}}\sqrt{t}\bigvee \frac{t^{3/2}}{B_r(\Sigma)\sqrt{n}}\right)+ e^{-t} + \frac{c''}{B_r(\Sigma)}, 
\end{align}
where we used the fact that $\Phi$ is a Lipschitz function with constant less than one. 
Quite similarly, 
\begin{align}
\label{B-Ess-2}
&
{\mathbb P}\left\{\frac{n}{B_r(\Sigma)}\xi \leq x\right\}
\geq 
{\mathbb P}\left\{\tau \leq x-c'\left(\frac{t}{\sqrt{n}}\bigvee \frac{{\bf r}(\Sigma)}{B_r(\Sigma)\sqrt{n}}\sqrt{t}\bigvee \frac{t^{3/2}}{B_r(\Sigma)\sqrt{n}}\right)\right\}-e^{-t} 
\\
&
\nonumber
\geq 
\Phi\left(x-c'\left(\frac{t}{\sqrt{n}}\bigvee \frac{{\bf r}(\Sigma)}{B_r(\Sigma)\sqrt{n}}\sqrt{t}\bigvee \frac{t^{3/2}}{B_r(\Sigma)\sqrt{n}}\right)\right)- e^{-t} - \frac{c''}{B_r(\Sigma)}
\\
&
\nonumber
\geq 
\Phi(x)-c'\left(\frac{t}{\sqrt{n}}\bigvee \frac{{\bf r}(\Sigma)}{B_r(\Sigma)\sqrt{n}}\sqrt{t}\bigvee \frac{t^{3/2}}{B_r(\Sigma)\sqrt{n}}\right)-e^{-t} - \frac{c''}{B_r(\Sigma)}. 
\end{align}
It follows from (\ref{B-Ess-1}) and (\ref{B-Ess-2}) that 
\begin{align}
\label{B-Ess}
&
\sup_{x\in {\mathbb R}}\left|{\mathbb P}\left\{\frac{n}{B_r(\Sigma)}\xi \leq x\right\}-\Phi(x)\right|
\\
&
\nonumber
\leq 
c'\left(\frac{t}{\sqrt{n}}\bigvee \frac{{\bf r}(\Sigma)}{B_r(\Sigma)\sqrt{n}}\sqrt{t}\bigvee \frac{t^{3/2}}{B_r(\Sigma)\sqrt{n}}\right)+  \frac{c''}{B_r(\Sigma)}+ e^{-t}.  
\end{align}
The last bound will be used with 
\begin{equation}
\label{t_def}
t=\log B_r(\Sigma)\bigwedge \log \left(\frac{B_r(\Sigma)\sqrt{n}}{{\bf r}(\Sigma)}\bigvee 2\right)\bigwedge \log n,
\end{equation}
which implies that 
\begin{equation}
\label{e-t}
e^{-t}\lesssim \frac{t}{\sqrt{n}}\bigvee \frac{{\bf r}(\Sigma)}{B_r(\Sigma)\sqrt{n}}\sqrt{t}\bigvee \frac{t^{3/2}}{B_r(\Sigma)\sqrt{n}}
\bigvee \frac{1}{B_r(\Sigma)}
\end{equation}
and we also have 
$$
\frac{t^{3/2}}{B_r(\Sigma)\sqrt{n}}\leq \frac{\log n}{\sqrt{n}}\frac{\sqrt{\log B_r(\Sigma)}}{B_r(\Sigma)}. 
$$
Without loss of generality we can assume that $B_r(\Sigma)$ is bounded away from $0$ by a numerical 
constant so that $\frac{\sqrt{\log B_r(\Sigma)}}{B_r(\Sigma)}\leq 1$ (otherwise, the bounds of the theorem
trivially hold). This implies that 
$
\frac{t^{3/2}}{B_r(\Sigma)\sqrt{n}}\leq \frac{\log n}{\sqrt{n}} 
$
and (\ref{B-Ess}) implies 
\begin{align}
\label{B-Ess_A}
&
\sup_{x\in {\mathbb R}}\left|{\mathbb P}\left\{\frac{n}{B_r(\Sigma)}\xi \leq x\right\}-\Phi(x)\right|
\\
&
\nonumber
\leq 
C\left[\frac{1}{B_r(\Sigma)}+\frac{{\bf r}(\Sigma)}{B_r(\Sigma)\sqrt{n}}\sqrt{\log \left(\frac{B_r(\Sigma)\sqrt{n}}{{\bf r}(\Sigma)}\bigvee 2\right) }+ \frac{\log n}{\sqrt{n}}\right],
\end{align}
which proves bound (\ref{normal_approx_A}).

To complete the proof of bound (\ref{normal_approx}), it is enough to use Theorem \ref{th:varvar} to replace the normalization with $\frac{n}{B_r(\Sigma)}$ by the normalization with the standard deviation of $\xi.$ To this end, note that 
\begin{align}
\label{golova}
&
\frac{\xi}{\sigma_{\xi}}= \frac{n}{B_r(\Sigma)}\xi + \left(\frac{1}{\sigma_{\xi}}-\frac{n}{B_r(\Sigma)}\right)\xi.
\end{align}
Under the assumptions $m_r\lesssim 1$ and $\|\Sigma\|_{\infty}\lesssim \bar g_r,$ we 
get from Theorem \ref{th:varvar} that 
$$
\left|\frac{n}{B_r(\Sigma)}\sigma_{\xi}-1\right|\lesssim \frac{{\bf r}(\Sigma)}{B_r(\Sigma)\sqrt{n}}+\frac{1}{n}.
$$
Without loss of generality, we can and do assume that $\frac{{\bf r}(\Sigma)}{B_r(\Sigma)\sqrt{n}}+\frac{1}{n}\leq c$
for a small enough constant $c>0$ so that  $\left|\frac{n}{B_r(\Sigma)}\sigma_{\xi}-1\right|\leq 1/2$
(otherwise, the bound of the theorem is trivial). Then
$$
\left|\left(\frac{1}{\sigma_{\xi}}-\frac{n}{B_r(\Sigma)}\right)\xi\right| 
\leq \left|\frac{n}{B_r(\Sigma)}\sigma_{\xi}-1\right|\frac{|\xi|}{\sigma_{\xi}}
\lesssim \left(\frac{{\bf r}(\Sigma)}{B_r(\Sigma)\sqrt{n}}+\frac{1}{n}\right) \frac{n}{B_r(\Sigma)}|\xi|.
$$
Combining this with bound of Theorem \ref{HS-conc}, we get that with probability at least $1-e^{-t}$
$$
\left|\left(\frac{1}{\sigma_{\xi}}-\frac{n}{B_r(\Sigma)}\right)\xi\right| 
\lesssim \left(\frac{{\bf r}(\Sigma)}{B_r(\Sigma)\sqrt{n}}+\frac{1}{n}\right)\left(\sqrt{t} \bigvee \frac{t}{B_r(\Sigma)}\bigvee 
\frac{{\bf r}(\Sigma)}{B_r(\Sigma)\sqrt{n}}\sqrt{t}\right).
$$
Using the last bound with $t$ defined by (\ref{t_def}), we easily get that 
\begin{equation}
\label{hvost}
\left|\left(\frac{1}{\sigma_{\xi}}-\frac{n}{B_r(\Sigma)}\right)\xi\right|
\lesssim 
\left[\frac{{\bf r}(\Sigma)}{B_r(\Sigma)\sqrt{n}}\sqrt{\log \left(\frac{B_r(\Sigma)\sqrt{n}}{{\bf r}(\Sigma)}\bigvee 2\right) }+ \frac{\log n}{\sqrt{n}}\right].
\end{equation}

The result now follows from (\ref{B-Ess}), (\ref{golova}) and (\ref{hvost}) by proving bounds on ${\mathbb P}\left\{\frac{\xi}{\sigma_{\xi}}\leq x\right\}$ similar to (\ref{B-Ess-1}), 
(\ref{B-Ess-2}).

\qed

\end{proof}

\section{Concluding remarks}

\noindent\textbf{1.} We start this section with deducing from the non-asymptotic bound of Theorem \ref{th:normal_approx} an asymptotic normality result.
To this end, consider a sequence of 
problems in which the data is sampled from Gaussian distributions in ${\mathbb H}$
with mean zero and covariance $\Sigma=\Sigma^{(n)}.$ 
Let $X=X^{(n)}$ be a centered Gaussian random vector in ${\mathbb H}$ with covariance operator $\Sigma=\Sigma^{(n)}$ and let  
$X_1=X_1^{(n)},\dots, X_n=X_n^{(n)}$ be i.i.d. copies of $X^{(n)}.$ The sample covariance based on $(X_1^{(n)},\dots, X_n^{(n)})$ is denoted by $\hat \Sigma_n.$ 
Let $\sigma(\Sigma^{(n)})$ be the spectrum of $\Sigma^{(n)},$ $\mu_{r}^{(n)}, r\geq 1$ be distinct nonzero eigenvalues of $\Sigma^{(n)}$ arranged in decreasing order and $P_{r}^{(n)}, r\geq 1$ be the corresponding spectral projectors.
As before, denote $\Delta_{r}^{(n)}:=\{j: \sigma_j(\Sigma^{(n)})=\mu_{r}^{(n)}\}$
and let $\hat P_{r}^{(n)}$ be the orthogonal projector on the direct sum 
of eigenspaces corresponding to the eigenvalues $\{\sigma_j(\hat \Sigma_n), j\in \Delta_{r}^{(n)}\}.$

Suppose that the spectral projector of $\Sigma^{(n)}$ to be estimated is $P^{(n)}=P_{r_n}^{(n)},$ the corresponding 
eigenvalue is $\mu^{(n)}=\mu_{r_n}^{(n)},$ its multiplicity is $m^{(n)}=m_{r_n}^{(n)}$ and its spectral gap is 
$\bar g^{(n)}=\bar g_{r_n}^{(n)}.$ 
%Our purpose is to prove that the sequence 
%$$
%\frac{n\Bigl(\|\hat P^{(n)}-P^{(n)}\|_2^2-{\mathbb E}\|\hat P^{(n)}-P^{(n)}\|_2^2\Bigr)}{B_n},
%$$
Denote  
$$
B_n:=B_{r_n}(\Sigma^{(n)}):= 2\sqrt{2}\|C^{(n)}\Sigma^{(n)}C^{(n)}\|_2\|P^{(n)}\Sigma^{(n)}P^{(n)}\|_2.
$$
%is asymptotically standard normal. 
The following assumption on $\Sigma^{(n)}$ will be needed:
 
\begin{assumption}
\label{ass_sigma_n}
Suppose the following conditions hold:
\begin{equation}
\label{ass_sigma_n_1}
\sup_{n\geq 1}m^{(n)}<+\infty\ \ {\rm and}\ \ 
\sup_{n\geq 1}\frac{\|\Sigma^{(n)}\|_{\infty}}{\bar g^{(n)}}<+\infty; 
\end{equation}

\begin{equation}
\label{ass_sigma_n_4}
B_n\to \infty\ \ {\rm and}\ \ \frac{{\bf r}(\Sigma^{(n)})}{B_n\sqrt{n}}\to 0\ {\rm as}\ n\goin.
\end{equation}
\end{assumption}

Note that Assumption \ref{ass_sigma_n}
implies that 
$
{\bf r}(\Sigma^{(n)})\to \infty\ {\rm and}\ {\bf r}(\Sigma^{(n)})=o(n)\ {\rm as}\ n\to\infty.
$
This easily follows from 
%using (\ref{ass_sigma_n_1}) and (\ref{ass_sigma_n_2}), 
$$
B_n\leq 2\sqrt{2} \sqrt{m^{(n)}}\biggl(\frac{\|\Sigma^{(n)}\|}{g^{(n)}}\biggr)^2
\sqrt{{\bf r}(\Sigma^{(n)})}=O\Bigl(\sqrt{{\bf r}(\Sigma^{(n)})}\Bigr)
$$
and (\ref{ass_sigma_n_4}).
%By (\ref{ass_sigma_n_3}),  this implies that ${\bf r}(\Sigma^{(n)})\to \infty$
%and, by (\ref{ass_sigma_n_4}), that ${\bf r}(\Sigma^{(n)})=o(n).$
It is also easy to see that, under mild further assumptions, 
$
B_n\asymp \|\Sigma^{(n)}\|_2.
$

\begin{corollary}
\label{th_2_norm}
Suppose Assumption \ref{ass_sigma_n} holds. 
Then
$$
\mathrm{Var}(\|\hat P^{(n)}-P^{(n)}\|_2^2) =\left( \frac{B_n}{n}\right)^2 (1+ o(1))
$$
and the sequences of random variables 
\begin{align}\label{CLT-sample proj}
\biggl\{\frac{n\Bigl(\|\hat P^{(n)}-P^{(n)}\|_2^2-{\mathbb E}\|\hat P^{(n)}-P^{(n)}\|_2^2\Bigr)}{B_n}\biggr\}_{n\geq 1}
\end{align}
and 
$$
\biggl\{\frac{\Bigl(\|\hat P^{(n)}-P^{(n)}\|_2^2-{\mathbb E}\|\hat P^{(n)}-P^{(n)}\|_2^2\Bigr)}
{\sqrt{\mathrm{Var}(\|\hat P^{(n)}-P^{(n)}\|_2^2)}}\biggr\}_{n\geq 1}
$$
both converge in distribution to the standard normal random variable. 
\end{corollary}

\textbf{2.} Neither normal approximation bounds of Theorem \ref{th:normal_approx}, nor the asymptotic normality result of Corollary   
\ref{th_2_norm} could be directly used to construct confidence regions for spectral projectors of covariance operators 
or to develop hypotheses tests. The reason is that, in these results, the squared Hilbert--Schmidt norm $\|\hat P^{(n)}-P^{(n)}\|_2^2$
is centered with its expectation and normalized with its standard deviation (or, alternatively, with $\frac{n}{B_r(\Sigma)}$)
that depend on unknown covariance operator $\Sigma.$ It would be of interest to develop ``data-driven" versions 
of these results, but this problem seems to be challenging and goes beyond the scope of the current paper. 
At the moment, we have only a partial solution (that is far from being perfect) of this problem in the case when the target spectral projector $P^{(n)}$ is one-dimensional (that is, the eigenvalue $\mu^{(n)}$ is of multiplicity one). We briefly outline 
such a result below. 
%We use the assumptions and the notations of Section \ref{Sec:CLT-HS}. 
Assume that we are given a sample of size $3n$ of i.i.d. centered Gaussian vectors 
$$\left\{X^{(n)}_1,\ldots, X_n^{(n)}; \tilde X^{(n)}_{1}, \ldots, \tilde X^{(n)}_n; \bar X^{(n)}_1,\ldots,\bar X^{(n)}_n  \right\},$$
with common covariance operator $\Sigma^{(n)}$. For each of the three subsamples of size $n,$ define its sample 
covariance operator: 
$$
\hat \Sigma^{(n)} =  \frac{1}{n}\sum_{i=1}^n X_i^{(n)} \otimes X_i^{(n)},\quad  \tilde \Sigma^{(n)} =  \frac{1}{n}\sum_{i=1}^n \tilde X_i^{(n)} \otimes\tilde X_i^{(n)},\quad \bar \Sigma^{(n)} =  \frac{1}{n}\sum_{i=1}^n \bar X_i^{(n)} \otimes \bar X_i^{(n)}.
$$
Let $\hat P^{(n)}$ be the orthogonal projector onto the eigenspace associated with the eigenvalue $\hat \mu^{(n)}$ of $\hat \Sigma^{(n)}$ (which is of multiplicity one with a high probability).
Similarly, $\tilde P^{(n)}$ and $\bar P^{(n)}$ are the orthogonal projectors onto the eigenspaces associated with the eigenvalue 
$\tilde \mu^{(n)}$ of $\tilde \Sigma^{(n)}$ and the eigenvalue $\bar \mu^{(n)}$ of $\bar \Sigma^{(n)},$ respectively. 
%We assume in addition in this section that $\mu^{(n)}$ is of multiplicity $1$ for any $n\geq 1$. 
Denote 
$$
\hat b^{(n)} = \sqrt{\bigl\langle \hat P^{(n)}, \tilde P^{(n)} \bigr\rangle} -  1\ \ {\rm and}\ \  \tilde b^{(n)} = \sqrt{\bigl\langle \tilde P^{(n)}, \bar P^{(n)} \bigr\rangle} -  1.
$$ 
%of the bias $b^{(n)}$.
%Therefore, we have under our assumptions, for $n$ large enough with probability equal to $1$, that $\hat \mu^{(n)}$, $\tilde \mu^{(n)}$ and $\bar \mu^{(n)}$ are also of multiplicity $1$.
It turns out that the statistic $-2\hat b^{(n)}$ can be used as an estimator of the expectation ${\mathbb E}\|\hat P^{(n)}-P^{(n)}\|_2^2$
while the statistic $\left|(1+\hat b^{(n)})^2 - (1+\tilde b^{(n)})^2\right|$ can be used to estimate the standard deviation  
${\rm Var}^{1/2}(\|\hat P^{(n)}-P^{(n)}\|_2^2)$ (note that $\hat b^{(n)}$ was introduced and studied in \cite{Koltchinskii_Lounici_bilinear}
as an estimator of a ``bias parameter" of  empirical spectral projectors and empirical eigenvectors).  
Moreover, it can be proved that, under Assumption \ref{ass_sigma_n},  
the sequence
\begin{align}\label{pure-data-driven-stat}
\left\lbrace  \frac{\|\hat P^{(n)}-P^{(n)}\|_2^2 + 2\hat b^{(n)} }{\left|(1+\hat b^{(n)})^2 - (1+\tilde b^{(n)})^2\right|}  \right\rbrace_{n\geq 1}
\end{align}
converges in distribution to a Cauchy type random variable.

For the spiked covariance model (\ref{spike}) with $m, s_1^2, \dots, s_m^2, \sigma^2$ being fixed 
and $p=p_n\to \infty$ as $n\to\infty,$ it is easy to find a simpler version of data-driven normalization 
with the limit distribution being standard normal.  For simplicity, assume that $m=1,$ so, the goal is 
to estimate the first principal components $\theta_1.$
Recall that in this case 
$
B_n=B_1(\Sigma^{(n)})=\frac{2\sqrt{2}(s_1^2+\sigma^2)\sigma^2\sqrt{p-1}}{s_1^4}
$
(see (\ref{B_r_large_p})).

Thus, the following estimator of $B_n$ could be used: $
\hat B_n = 2\sqrt{2} \frac{\hat\mu_1^{(n)} \hat\mu_2^{(n)}}{(\hat\mu_1^{(n)} - \hat \mu_2^{(n)})^2} \sqrt{p_n-1},
$
where $\hat \mu_1^{n}$ and $\hat\mu_2^{(n)}$ are the largest and the second largest eigenvalues of $\hat\Sigma^{(n)} = \frac{1}{n}\sum_{i=1}^{n} X_i^{(n)}\otimes X_i^{(n)},$ respectively. 
In the case of such a spiked covariance model, Assumption \ref{ass_sigma_n} is equivalent to $p=p_n\to \infty$ and $p=o(n).$ 
Under these assumptions, it is easy to prove that
$
\frac{\hat B_n}{n} 
%= \frac{B_n}{n} + O_{\mathbb P}\left((s_1^2  + \sigma^2) \sqrt{\frac{p_n}{n}} \right) 
= \frac{B_n}{n}\left(1 + o_{\mathbb P}(1)\right).
$

Let $P_1 = \theta_1\otimes \theta_1.$ Then, it can be proved that  
the sequence
\begin{align}\label{data-driven-stat}
\left\{ \frac{n}{\hat B_n} \left( \|\hat P^{(n)}_1-P^{(n)}_1\|_2^2  + 2 \hat b^{(n)}\right)\right\}_{n\geq 1}
\end{align}
converges in distribution to a standard normal random variable.

\noindent \textbf{3.} To illustrate the asymptotic behavior of standard PCA, we consider the following spiked covariance setting. Let $X_1,\ldots, X_n,\tilde X_1,\ldots, \tilde X_n, \bar X_1,\dots, \bar X_n$ be $3n$ i.i.d. random vectors in $\R^p$ with covariance $\Sigma = s_1^2 (\theta_1 \otimes \theta_1) +  \sigma^2 I_p$, $s_1^2 = 2$, $\sigma^2 = 1/10,$ where $\theta_1$ is an arbitrary unit vector in $\R^p$. For selected values of $(n,p),$ we computed the statistic $\|\hat P^{(n)}_1 - P_1 \|_2^2$, $\hat B_n$ and the empirical bias estimators $\hat b^{(n)}_1$, $\tilde b_1^{(n)}$ as well as the statistics (\ref{pure-data-driven-stat}) and 
\begin{align}\label{theory-stat}
\left\{ \frac{n}{ B_n} \left( \|\hat P^{(n)}_1-P_1\|_2^2  + 2 \hat b^{(n)}\right)\right\}_{n\geq 1}.
\end{align}
We performed $1000$ replications of this experiment. 

\noindent In Table \ref{table-1-mean-risk}, we compare the sample mean of the statistic $\|\hat P^{(n)}_1 - P_1 \|_2^2$ denoted by $\hat m_n$ (that provides an estimator of the risk $\E\|\hat P_1^{(n)} -P_1\|_2^2$ based on the repeated 
samples of size $n$) to the estimated risk $-2\hat b^{(n)}_1$ for each individual sample and the first order approximation of the theoretical risk derived in (\ref{Risk-mean}) which can be computed easily in this model since $A_n := A_1(\Sigma) = 2\frac{(s_1^2+\sigma^2)\sigma^2}{s_1^4} (p-1)$. More precisely, in the second row of the table the sample means of
$\frac{|2\hat b_1^{(n)}+\hat m_n|}{|\hat m_n|}$ over $1000$ replications of the experiment are presented.   
The results show that $-2\hat b^{(n)}_1$ provides a somewhat better approximation of the risk $\E\|\hat P_1^{(n)} -P_1\|_2^2$ than the first order approximation (\ref{Risk-mean}) for small sample size. For relatively large sample size, the first order approximation (\ref{Risk-mean}) becomes more precise than the estimator $-2\hat b^{(n)}_1$.

\begin{table}[!htbp]
\begin{tabular}{|c|c|c|c|c|c|c|c|c|c|}
\hline
 $n$ &\thead{$100$}&\thead{$200$}&\thead{$300$}&\thead{$500$}&\thead{$10^3$}&\thead{$10^4$}\\ \hline
$\frac{| A_n/n - \hat m_n|}{|\hat m_n|}$ &$0.49$ & $0.24$ &  $0.15$ & $0.1$ &  $ 0.049$ & $0.008$ \\   
\hline
$\frac{|2 \hat b_1^{(n)} + \hat m_n|}{|\hat m_n|}$& $0.07 $ & $ 0.06$ &$ 0.054$ & $0.052$ & $0.045$ & $ 0.036$ \\ 
\hline
\end{tabular}
\caption{Relative deviation of the risk approximation $\frac{A_n}{n}$ and the risk estimator $-2 \hat b_1^{(n)}$ from the sample risk  $\hat m_n$ for $p=10^3$.}\label{table-1-mean-risk}
\end{table}

\noindent In Table \ref{table-2-variance-risk}, we compare the sample variance of the statistic $\|\hat P^{(n)}_1 - P_1 \|_2^2$ denoted by $\hat S_n^2$ to the variance estimator $\tilde V_n:=\left((1+\hat b_1^{n})^2- (1+\tilde b_1^{(n)})^2\right)^2$ and also to the first order approximation of the theoretical variance $\frac{B_n^2}{n^2}$ derived in (\ref{Risk-var}) with $B_n= 2 \sqrt{2} \frac{(s_1^2+\sigma^2)\sigma^2}{s_1^4} \sqrt{p-1}$. Again, in the second row of the table the sample means of 
$\frac{|\tilde V_n-\hat S_n^2|}{\hat S_n^2}$ over $1000$ replications of the experiment are presented.
We observe that $\tilde V_n$ and $\frac{B_n^2}{n^2}$ provide reasonable approximation of the variance of $\|\hat P^{(n)}_1 - P_1\|_2^2$ only for relatively large sample sizes.

\begin{table}[!htbp]
\begin{tabular}{|c|c|c|c|c|c|c|c|c|c|}
\hline
 $n$ &\thead{$100$}&\thead{$200$}&\thead{$300$}&\thead{$500$}&\thead{$10^3$}&\thead{$10^4$}\\ \hline
$\frac{| B_n^2/n^2 - \hat S_n^2 |}{\hat S_n^2 }$ &$0.62$ & $0.65$ & $0.66$& $0.58$ & $0.42$ & $0.07$ \\   
%\hline
%$\frac{|\hat B_n^2/n^2 - \hat S_n^2 |}{\hat S_n^2 }$& $ 0.10$ & $0.015$ & $6.2\cdot 10^{-3} $ & $2.9\cdot 10^{-3}$ & $ 1.2\cdot 10^{-3}$ & $3\cdot 10^{-4} $\\ 
\hline
$\frac{|\tilde V_n - \hat S_n^2 |}{\hat S_n^2 }$& $ 0.82$ & $0.73$ & $0.67 $ & $0.58$ & $ 0.39$ & $0.05 $\\
\hline

\end{tabular}
\caption{Relative deviation of the variance estimator $\tilde V_n$ and the variance approximation $\frac{B_n^2}{n^2}$ from the sample variance $\hat S_n^2$ for $p=10^3$.}\label{table-2-variance-risk}
\end{table}

%$\hat B_n^2/n^2$ and

\noindent Finally, we compute empirical densities of the statistics (\ref{pure-data-driven-stat}) and (\ref{theory-stat}) and compare them with their respective theoretical limiting distributions in Figure \ref{normal convergence}. For (\ref{theory-stat}), we also provide the empirical mean and variance.

\begin{figure}[!htbp]
\vspace{5mm}
\hspace{-7cm}
\includegraphics[width=0.7\linewidth,natwidth=423,natheight=343]{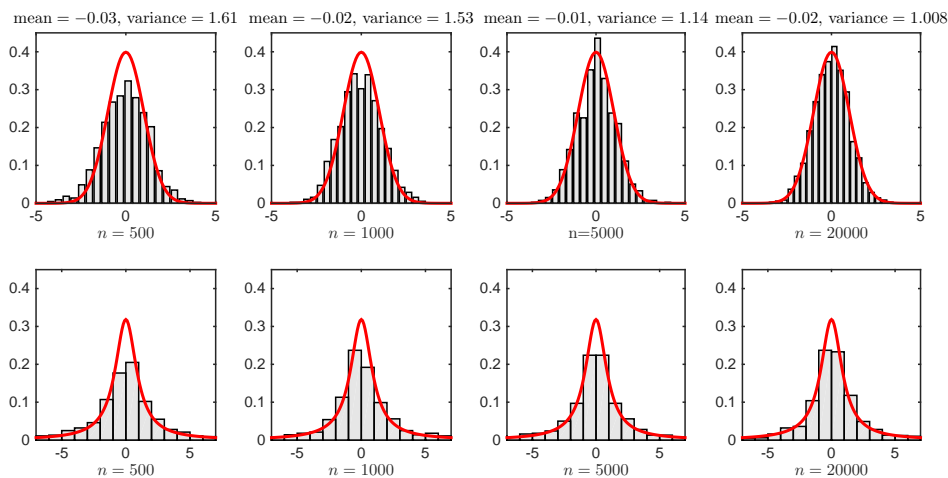}
\caption{Top: empirical distribution of (\ref{theory-stat}) and standard normal density for $p=1000$. Bottom: empirical distribution and theoretical Cauchy distribution of (\ref{pure-data-driven-stat}) for $p=1000$.}\label{normal convergence}
\end{figure}

%
%\bibliographystyle{plain}
%\bibliography{biblio}

\end{document}